 \documentclass[11pt]{amsart}

\theoremstyle{plain}
\newtheorem{theorem}{Theorem}[section]
\newtheorem{lemma}[theorem]{Lemma}
\newtheorem{proposition}[theorem]{Proposition}
\newtheorem{corollary}[theorem]{Corollary}

\theoremstyle{definition}
 \theoremstyle{remark}
\newtheorem{remark}[theorem]{Remark}

\begin{document}
\newcommand{\RR}{{\mathbb R}}
 \newcommand{\N}{{\mathbb N}}
\newcommand{\Rn}{{\RR^n}}
\newcommand{\ieq}{\begin{equation}}
\newcommand{\eeq}{\end{equation}}
\newcommand{\ieqa}{\begin{eqnarray}}
\newcommand{\eeqa}{\end{eqnarray}}
\newcommand{\ieqas}{\begin{eqnarray*}}
\newcommand{\eeqas}{\end{eqnarray*}}
\newcommand{\Bo}{\put(260,0){\rule{2mm}{2mm}}\\}

\numberwithin{equation}{section}

\def\at#1{{\bf #1}: }
\def\att#1#2{{\bf #1}, {\bf #2}: }
\def\attt#1#2#3{{\bf #1}, {\bf #2}, {\bf #3}: }
\def\atttt#1#2#3#4{{\bf #1}, {\bf #2}, {\bf #3},{\bf #4}: }
\def\aug#1#2{\frac{\displaystyle #1}{\displaystyle #2}}
\def\figura#1#2{
\begin{figure}[ht]
\vspace{#1}
\caption{#2}
\end{figure}}
\def\B#1{\bibitem{#1}}
\def\q{\int_{\Omega^\sharp}}
\def\z{\int_{B_{\bar{\rho}}}\underline{\nu}\nabla (w+K_{c})\cdot
\nabla h}
\def\a{\int_{B_{\bar{\rho}}}}
\def\b{\cdot\aug{x}{\|x\|}}
\def\n{\underline{\nu}}
\def\d{\int_{B_{r}}}
\def\e{\int_{B_{\rho_{j}}}}
\def\LL{{\mathcal L}}
\def\D{{\mathcal D}}
\def\itr{\mathrm{Int}\,}
\def\tg{\tilde{g}}
\def\A{{\mathcal A}}
\def\S{{\mathcal S}}
\def\H{{\mathcal H}}
\def\M{{\mathcal M}}
\def\MM{{\mathcal M}}
\def\T{{\mathcal T}}
\def\N{{\mathcal N}}
\def\I{{\mathcal I}}
\def\F{{\mathcal F}}
\def\J{{\mathcal J}}
\def\E{{\mathcal E}}
\def\P{{\mathcal P}}
\def\HH{{\mathcal H}}
\def\V{{\mathcal V}}
\def\B{{\mathcal B}}
\newcommand{\average}{{\mathchoice {\kern1ex\vcenter{\hrule
height.4pt width 6pt depth0pt}
\kern-11pt} {\kern1ex\vcenter{\hrule height.4pt width 4.3pt
depth0pt} \kern-7pt} {} {} }}
\newcommand{\ave}{\average\int}
\title[Radial fractional Laplace operators and Hessian inequalities]{Radial fractional Laplace operators\\ and Hessian inequalities}


\author[F. Ferrari]
{Fausto Ferrari}
\address{
 Dipartimento di Matematica, dell'Universit\`a di Bologna, Piazza di Porta S. Donato, 5, 40126, Bologna, Italy}
\email{fausto.ferrari@unibo.it}

\author[I. E. Verbitsky]
{Igor E. Verbitsky}
\address{Department of Mathematics,
University of Missouri,
Columbia, MO 65211, USA}
\email{verbitskyi@missouri.edu}

\date{\today}

%


\begin{abstract}
In this paper we deduce a formula for the fractional Laplace operator 
 $(-\Delta)^{s}$ on  radially symmetric functions useful for  some applications. We  give a criterion of subharmonicity associated with $(-\Delta)^{s}$, and  apply it  to 
a problem related to the Hessian inequality of Sobolev type: 
$$\int_{\mathbb{R}^n}\left \vert (-\Delta)^{\frac{k}{k+1}} u\right \vert^{k+1} dx \le 
C \, \int_{\mathbb{R}^n} - u \, F_k[u] \, dx,
$$
where $F_k$ is the $k$-Hessian operator on $\mathbb{R}^n$, $1\le k < \frac n 2$, under some restrictions on a $k$-convex function $u$. In particular, we show that the class of $u$ for which the above inequality was 
established in \cite{FFV} contains the extremal functions for the Hessian Sobolev 
inequality of X.-J. Wang \cite{W1}. This is proved using  logarithmic convexity of the Gaussian ratio of 
hypergeometric functions which might be of independent interest. 

\end{abstract}

\maketitle

\section{Introduction}

 Let $n \ge 2$. For every $s\in (0,1)$ and a locally integrable function $u:\mathbb{R}^n\to\mathbb{R}$ such that 
 \begin{equation}\label{integrability}
\int_{\mathbb{R}^n}\frac{\mid u(x)\mid}{(1+\mid x\mid^2)^{\frac{n+2s}{2}}}dx<+\infty,
\end{equation}
we  define the $s$-fractional Laplace operator as follows: 
\begin{equation}\label{lap-def}
(-\Delta)^{s}u(x)=c_{s,n} \,\int_{\mathbb{R}^n}\frac{u(x)-u(y)}{\mid x-y\mid^{n+2s}}dy,
\end{equation} 
where $c_{s,n}$ is a positive normalization constant. The integral 
is convergent  (in the principal value sense if $\tfrac 12 \le s<1$) if $u$ is, for instance,  a bounded $C^2$ function. For any  $u$ satisfying (\ref{integrability}), $(-\Delta)^{s} u$ is defined in the sense of distributions:
 \begin{equation}\label{distr}
\langle (-\Delta)^{s} u, \, h \rangle = \int_{\mathbb{R}^n} u \, (-\Delta)^{s} \, h \, dx, 
\end{equation}
for all test functions $h \in C^\infty_0 (\mathbb{R}^n)$.  This is  a linear nonlocal operator, that is, roughly  speaking, if $u$ satisfies $(-\Delta)^{s}u(x)=0$ in a domain $\Omega\subseteq \mathbb{R}^n,$  then the value of $u$ at any point of $\Omega$ depends not  only on the neighborhood of the point itself but also on the behavior of the function in the entire space. 
We refer to \cite{Lan} and \cite{CaS} for further details.

Some new relationships with a class of local nonlinear operators, the $k$-th Hessian operators,  have been pointed out in  \cite{FFV}. We recall  that the $k$-th Hessian operator $F_k$ can be defined in several equivalent ways (see e.g. \cite{W2}). For a $C^2$ function $u$ in a domain $\Omega\subset \mathbb{R}^n$,  we define $F_k[u]$, $k\in \mathbb{N},$ $1\leq k\leq n,$ as the $k$-th symmetric elementary function of the eigenvalues of $D^2u(x),$ the Hessian matrix of $u.$ Just to give an idea of this family of operators we note that $F_1[u]=\Delta u$ when $k=1$, and  $F_n[u]=\mbox{det}(D^2u)$ 
when  $k=n$. A function $u \in C^2(\Omega)$ is called $k$-convex if 
$F_j[u]\ge 0$ for all $j=1, 2, \ldots, k$. This definition was extended to general upper semicontinuous functions by Trudinger and Wang \cite{TW1}, \cite{TW2}. In particular, $u$ is $1$-convex if and only if $u$ is subharmonic, while $u$ is $n$-convex if 
it is convex in the usual sense. 

Clearly, $F_k$ are local operators, since it is possible to calculate $F_k[u]$ pointwise by using second order partial derivatives of $u$.  Moreover, $k$-th Hessian operators are fully nonlinear while 
$s$-fractional Laplace operators are linear. Nevertheless, 
these operators are closely related. It was shown in \cite{FFV} 
(Theorem 2.1) 
that, under certain assumptions on $u$ discussed below, the following inequality holds (see  Proposition~\ref{direct} in this paper for the detailed statement):

\begin{equation}\label{eq-2.int} 
\int_{\mathbb{R}^n}\left \vert (-\Delta)^{s}u\right \vert^{k+1}dx\leq C_{k, n} \, \int_{\mathbb{R}^n}-u \, F_k[u] \, dx,
\end{equation}
where $s=\frac{k}{k+1}$, $1 \le k < \frac n 2$, and   $C_{k, n}$ is 
positive constant depending only on $k$ and $n$. 
The converse inequality holds as well with a different constant, for all $k$-convex $C^2$ functions $u$ vanishing at infinity (\cite{FFV}, Theorem 3.1). 

In this paper we want to improve our knowledge of the $s$-fractional Laplace operator by studying $s$-subharmonic functions that are analogous to subharmonic functions for the Laplace operator, and in particular the class of $k$-convex 
functions introduced in \cite{FFV} for which (\ref{eq-2.int}) holds. In fact we will prove below that this class contains extremal functions of the 
Hessian Sobolev inequality due to 
X.-J. Wang \cite{W1}, \cite{W2}:  
\begin{equation}\label{hessian-sobolev}
\left (\int_{\mathbb{R}^n}\left \vert u\right \vert^{q}dx\right)^{\frac {1}{q}}\leq C'_{k, n} \,\left ( \int_{\mathbb{R}^n}-u \, F_k[u] \, dx\right)^{\frac {1}{k+1}},
\end{equation}
where $1\le k < \frac n 2$, $q=\frac{n(k+1)}{n-2k}$, and $u$ is a $k$-convex $C^2$ function on $\mathbb{R}^n$ vanishing at $\infty$. We remark that (\ref{eq-2.int}) implies (\ref{hessian-sobolev}) by the classical Sobolev embedding theorem. 

 Our approach is to study  the $s$-fractional Laplace operator 
 and the corresponding notion of $s$-subharmonicity 
 for radially symmetric functions $u(x)=u(r)$, $r=|x|$. We will prove the following result.  
\begin{theorem}\label{radialrappresentation}
Let $s\in (0,1).$ For every radial $C^2$ function  $u$ such that
\begin{equation}\label{ineFFV}
\int_{0}^{+\infty}\frac{\mid u(r)\mid}{(1+r)^{n+2s}} \, r^{n-1} \, dr<+\infty,
\end{equation}
 the following formula holds: 
 \begin{equation}\begin{split}\label{rapresth}
(-\Delta)^su(r)= c_{s,n} \, r^{-2s}\int_1^{+\infty}\left(u(r)-u(r\tau)+(u(r)-u(\frac{r}{\tau}))\tau^{-n+2s}\right)\tau( \tau^2-1
  )^{-1-2s} H(\tau)d\tau,
\end{split}
\end{equation}
 where $r=|x|>0$, $x \in \mathbb{R}^n$, and 
$$
H(\tau)=2\pi\alpha_n\int_{0}^{\pi}\sin^{n-2}\!\theta \, \frac{ (  \sqrt{\tau^2-\sin^2\theta}+\cos\theta)^{1+2s}}{\sqrt{\tau^2-\sin^2 \theta}} \, d\theta, \quad \tau \ge 1, \quad 
\alpha_n= \frac{\pi^{\frac{n-3}{2}}} { \Gamma ( \frac{n-1}{2} )}.
$$
\end{theorem}
Clearly, $H(\tau)$ is a positive continuous function on $[1, +\infty)$, with 
$H(\tau)\simeq \tau^{2s}$ as $\tau \to +\infty$. One can express $H(\tau)$ in terms of the Gaussian hypergeometric function: 
$$
\tau( \tau^2-1
  )^{-1-2s} H(\tau) = \frac{2 \pi^{\frac{n}{2}}} {\Gamma(\frac{n}{2})} \, 
  \tau^{-1-2s}  \, {_{2}}F_1(a, b, c, \tau^{-2})  \simeq (\tau-1)^{-1-2s},  \quad \tau \ge 1,
$$
where
$$
a=\frac{n+2s}{2}, \quad b=1+s, \quad c=\frac{n}{2}. 
$$
Note that  $_{2}F_1(a, b, c, \tau^{-2})$ has a singularity of order 
$(\tau-1)^{-1-2s}$ at $\tau=1$, while $H(\tau)$ is continuous at $\tau=1$. The integral in (\ref{rapresth}) is convergent for 
any radial $C^2$ function $u$ satisfying (\ref{ineFFV}). 

For every $s\in (0,1]$ and $u$  satisfying 
(\ref{integrability}),  
we  say that $u$ is $s$-subharmonic in  $\Omega$ if  
$u: \Omega \to [-\infty, +\infty)$ is upper semicontinuous in 
$\Omega$, and 
$$
- (-\Delta)^s u(x)\geq 0 \quad \text{in} \, \, D'(\Omega);
$$ 
in other words, $- (-\Delta)^s u$ is a positive Borel measure in $\Omega$.

\begin{theorem}\label{charachter_s_subharmonic_int}  Let $s \in (0, 1]$. Let $u(r)$  be a radial upper semicontinuous   function  in $\mathbb{R}_+$ such that (\ref{ineFFV}) holds.  If, for all $r>0$ and  $\tau \geq 1$, 
\begin{equation}\label{cond-1}
u(r)-u(r\tau)+(u(r)-u(\frac{r}{\tau}))\tau^{-n+2s}\leq 0,
\end{equation} 
 then  $u$ is $s$-subharmonic in $\mathbb{R}^n\setminus\{0\}$. 
\end{theorem}
If in Theorem~\ref{charachter_s_subharmonic_int} one assumes   
additionally that $u$ is bounded above in a neighborhood of $0$, or more generally $\limsup_{r\to 0} u(r) \, r^{n-2s}\le 0$, then $u$ has  a removable singularity at $0$ (see [BH], p. 379, Corollary 10.2), i.e.,  $u$ is $s$-subharmonic in the entire space $\mathbb{R}^n$ (with $u(0)=\limsup_{r\to 0} u(r)$). \medskip 

We observe that condition (\ref{cond-1}) is of interest for any real $s$. 
The following theorem provides a  more convenient  pointwise   
characterization of (\ref{cond-1}). Combined with 
Theorem~\ref{charachter_s_subharmonic_int}, it gives a useful sufficient condition for a radially symmetric function to be  $s$-subharmonic when $0<s<1$.

\begin{theorem}\label{s-subharm}
Let $s\in \mathbb{R}.$ Let $u \in C^2 (\mathbb{R}_+)$ be a radial function. Then  
 (\ref{cond-1}) holds if and only if, for all $r > 0$,  
 \begin{equation}\label{cond-2}
 u''(r) + (n-2s +1) \frac{ u'(r)}{r}\ge 0. 
 \end{equation} 
\end{theorem}

 Condition (\ref{cond-2}) is not necessary  for a radially symmetric function $u$ to be $s$-subharmonic 
 when $0<s<1$ (see examples below). However, the class of functions 
obeying (\ref{cond-2}) is quite rich, and contains many interesting 
 $s$-subharmonic functions. 
 
 In the special case $s=\frac{n(k-1)}{2k}+1$, condition (\ref{cond-2})  coincides with the 
following characterization of  $k$-convex radially symmetric $C^2$ 
functions (see \cite{W1}, \cite{W2}):
 \begin{equation}\label{cond-3}
 u''(r) + \frac{n-k}{k} \, \frac{ u'(r)}{r} \ge 0, 
 \end{equation} 
 where $k=1,2, \ldots, n$.

For $\beta>0$,  let us consider
  $$
f_\beta(x)=-(1+\mid x\mid^2)^{-\frac{\beta}{2}}, \quad x \in \mathbb{R}^n.
$$
 It is easy to see (cf. Corollary \ref{quasifunda_s_subh_int} below) that for any $s\in \mathbb{R}$, $f_\beta$ satisfies  (\ref{cond-1}), and consequently   (\ref{cond-2}), if and only if $\beta\le n- 2s$. Moreover, if $0<s\le1$ and $0<\beta<n$, then $f_\beta$ is 
$s$-subharmonic in $\mathbb{R}^n$ if and only 
if $\beta\le n- 2s$.

Now let $\beta=\frac n k -2>0$, where $k=1,2, \ldots, \left[ \frac{n}{2}\right]$. Then by (\ref{cond-3}), 
 $f_\beta$ is $k$-convex, and is known to be an extremal function 
 for the Hessian Sobolev inequality (\ref{hessian-sobolev}); see 
 \cite{W1}, \cite{W2}. 
 (A new proof of this inequality, but without the  sharp constant, 
 is available; see  \cite{V}). Our main goal is to show that $u=f_\beta$ 
 satisfies  the following condition introduced in \cite{FFV}: 
   \begin{equation}\label{iterated-intro}
(-\Delta)^{s}[-(-\Delta)^{s}u]^k\geq 0 \quad \text{in} \, \, \mathbb{R}^n,
    \end{equation} 
   where $s=\frac{k}{k+1}$, $1\le k < \frac {n}{2}$.   Under this condition  the enhanced Hessian Sobolev 
   inequality (\ref{eq-2.int}) was proved in \cite{FFV} (Theorem 2.1).    The following theorem  demonstrates  that, in particular,  this class of functions is nontrivial.
   
\begin{theorem}\label{k-convex-example} Let $s=\frac{k}{k+1}$ and $\beta= 
\frac n k -2$, where $1\le k \le \frac n 2$. 
Then 
$f_\beta$ satisfies condition (\ref{iterated-intro}). 
\end{theorem} 

The proof of Theorem~\ref{k-convex-example} for $k \ge 2$ relies on 
Theorems~\ref{charachter_s_subharmonic_int} and \ref{s-subharm}, along with a new logarithmic convexity property of the Gaussian 
quotient of hypergeometric functions $\frac{_{2}F_1(a,b,c,x)}{_{2}F_1(a,b+1,c+1,x)}$. 

It is worth observing that  when $k=1$, the function $-u= (-\Delta)^{\frac 1 2} f_\beta$ with $\beta=n-2$ ($n \ge 3$), fails to satisfy condition (\ref{cond-2}) with $s=\frac 1 2$; however, 
$-u$ is still $\frac 1 2$-subharmonic. The latter 
 is easily checked using the superposition 
property of fractional Laplacians: 
$$
(-\Delta)^{\frac 1 2}[-(-\Delta)^{\frac 1 2} f_\beta] = \Delta f_\beta \ge 0.  
$$ 

The paper is organized as follows. The proofs of Theorems \ref{radialrappresentation} and \ref{charachter_s_subharmonic_int} are given in Section \ref{proofofradialrapresentation}.  In Section \ref{appendixx} we prove Theorem 1.3 and discuss the  $s$-subharmonicity 
property for radially symmetric functions. Some examples 
are given in Section~\ref{examples_remarks}. In Section \ref{applicationsxx} we discuss some further applications of our results. In particular, we prove a Liouville theorem, 
the maximum principle for radial $s$-subharmonic functions, and  a derivative  formula involving the 
fractional Laplacian. 
In Section \ref{fractional-hessian}, we prove Theorem~\ref{k-convex-example} using a convexity property 
of a certain ratio of hypergeometric functions.  Finally, in Section \ref{logconvexity}, we prove logarithmic convexity of 
$\frac{_{2}F_1(a,b,c,x)}{_{2}F_1(a,b+1,c+1,x)}$ and convexity of $\frac{_{2}F_1(a,b,c,x)}{_{2}F_1(a+1,b+1,c+1,x)}$  in $(-\infty, 1)$ (Theorem~\ref{log-conv} and Corollary~\ref{conv2}) under 
certain restrictions on the parameters $a, b, c$ 
using a method developed recently 
in \cite{KS1} (see also \cite{KS2}).  The proofs generalize to ratios   of 
generalized hypergeometric functions  $_{q+1}{\!}F_{q}$ for $q\ge 2$.

\section{Proofs of Theorems \ref{radialrappresentation} 
and \ref{charachter_s_subharmonic_int}}\label{proofofradialrapresentation}

 \begin{proof}[Proof of Theorem~\ref{radialrappresentation}]   
For the sake of convenience, we will drop the normalization constant $c_{s,n}$ in (\ref{lap-def}) when it does 
not lead to a confusion. Let $u$ be a radial $C^2$ function in 
$\mathbb{R}^n$ satisfying (\ref{ineFFV}). If $u(x)=u(|x|),$  $x=r x',$  $y=\rho y',$  and $\mid x'\mid=\mid y'\mid=1,$ then
\begin{equation*}
\begin{split}
&(-\Delta)^{s}u(x)=\int_0^{+\infty}\left(\int_{\mid y'\mid=1}\frac{u(r)-u(\rho)}{\mid rx'-\rho y'\mid^{n+2s}}\rho^{n-1}dH^{n-1}(y')\right)d\rho\\
&=\int_0^{+\infty}(u(r)-u(\rho))\rho^{n-1}\left(\int_{\mid y'\mid=1}\frac{1}{\mid rx'-\rho y'\mid^{n+2s}}dH^{n-1}(y')\right)d\rho\\
&=\int_0^{+\infty}\frac{u(r)-u(\rho)}{r^{n+2s}}\rho^{n-1}\left(\int_{\mid y'\mid=1}\frac{1}{\mid x'-\frac{\rho}{r} y'\mid^{n+2s}}dH^{n-1}(y')\right)d\rho\\
&=r^{-1-2s}\int_0^{+\infty}(u(r)-u(r\tau))\tau^{n-1}r\left(\int_{\mid y'\mid=1}\frac{1}{\mid x'-\tau y'\mid^{n+2s}}dH^{n-1}(y')\right)d\tau\\
&=\int_0^{+\infty}\frac{u(r)-u(r\tau)}{r^{2s}}\tau^{n-1}\left(\int_{\mid y'\mid=1}\frac{1}{\mid x'-\tau y'\mid^{n+2s}}dH^{n-1}(y')\right)d\tau.
\end{split}
\end{equation*}

Notice that
$$
\int_{\mid y'\mid=1}\frac{1}{\mid x'-\tau y'\mid^{n+2s}}dH^{n-1}(y')
$$
is independent of $x'\in\{\mid y\mid=1\}$. Indeed,  suppose that $z' \in\{\mid y\mid=1\}.$ Then there exists a unitary matrix $Q$ such that $z'=Qx'.$ Thus, performing a change of variables such that $y'=Qw',$ we get:
\begin{equation*}
\begin{split}
\int_{\mid y'\mid=1}\frac{1}{\mid z'-\tau y'\mid^{n+2s}}dH^{n-1}(y')&=\int_{\mid w'\mid=1}\frac{1}{\mid Qx'-\tau Qw'\mid^{n+2s}}dH^{n-1}(w')\\
&=\int_{\mid w'\mid=1}\frac{1}{\mid x'-\tau w'\mid^{n+2s}}dH^{n-1}(w'),
\end{split}
\end{equation*}
because $\mid \mbox{det}Q\mid=1$ and $\mid Qv_1-Qv_2\mid=\mid v_1-v_2\mid$ for every $v_1,v_2\in \mathbb{R}^n.$

Moreover, 
\begin{equation*}
\begin{split}
\int_{\mid y'\mid=1}\frac{1}{\langle z'-\tau y',z'-\tau y'\rangle^{\frac{n+2s}{2}}}dH^{n-1}(y')&=\int_{\mid y'\mid=1}\frac{1}{(1-2\tau\langle y',z'\rangle+\tau^2)^{\frac{n+2s}{2}}}dH^{n-1}(y')\\
&=2\pi \alpha_n\int_{0}^{\pi}\frac{\sin^{n-2}\theta}{(1-2\tau\cos\theta+\tau^2)^{\frac{n+2s}{2}}}d\theta,
\end{split}
\end{equation*}
where
$$
\alpha_n=\prod_{k=1}^{n-3}\int_0^{\pi}\sin^k\theta d\theta= \frac{\pi^{\frac{n-3}{2}}} { \Gamma ( \frac{n-1}{2} )}. 
$$
Notice also that
$$
1-2\tau\cos\theta+\tau^2\geq 1-2\tau+\tau^2=(1-\tau)^2.
$$
Thus,
for every $\tau \not =1$, 
$$
 \int_{0}^{\pi}\frac{\sin^{n-2}\theta}{(1-2\tau\cos\theta+\tau^2)^{\frac{n+2s}{2}}}d\theta
 $$
is bounded, while for $\tau=1$, 
$$\frac{\sin^{n-2}\theta}{(2-2\cos\theta)^{\frac{n+2s}{2}}}
 \sim \theta^{n-2}\theta^{-n-2s}\sim \theta^{-2-2s}, \quad \text{as}  \, \, \theta \to 0.
$$

 We denote
  \begin{equation}\label{k-tau}
 K(\tau)=2\pi \alpha_n\int_{0}^{\pi}\frac{\sin^{n-2}\theta}{(1-2\tau\cos\theta+\tau^2)^{\frac{n+2s}{2}}}d\theta, 
  \end{equation}
 and study the behavior of the integral with respect to $\tau.$

 Let $\tau \ge 1$. We perform a change of variable as follows: 
 $$
 \frac{\sin \theta}{\sqrt{\tau^2-2\tau\cos \theta +1}}=
 \frac{\sin\psi} {\tau}.
 $$
 Consequently, 
 $$
 \cos\theta=\frac{\sin^2\psi\pm\cos\psi\sqrt{\tau^2-\sin^2\psi}}{\tau}.
 $$
 Hence
 \begin{equation}\begin{split}\label{deripsi}
-\sin\theta\frac{d\theta}{d\psi}&=\tau^{-1}\sin\psi\left(2\cos\psi\mp\sqrt{\tau^2-\sin^2\psi}\mp\frac{\cos^2\psi}{\sqrt{\tau^2-\sin^2\psi}}\right)\\
&=\frac{\sin\psi}{\tau\sqrt{\tau^2-\sin^2\psi}}\left(\cos\psi\mp\sqrt{\tau^2-\sin^2\psi}\right)^2
\end{split}
\end{equation}
Moreover
\begin{equation}\begin{split}\label{thetasintheta}
\tau\sin\theta &=\sin\psi\sqrt{\tau^2-2\tau\cos\theta+1}\\
&=\sin\psi\sqrt{\tau^2-2(\sin^2\psi\pm\cos\psi\sqrt{\tau^2-\sin^2\psi})+1}\\
&=\sin\psi\sqrt{\tau^2-\sin^2\psi\mp 2\cos\psi\sqrt{\tau^2-\sin^2\psi}+1-\sin^2\psi}\\
&=\sin\psi\sqrt{\tau^2-\sin^2\psi\mp 2\cos\psi\sqrt{\tau^2-\sin^2\psi}+\cos^2\psi}\\
&=\sin\psi\sqrt{(\sqrt{\tau^2-\sin^2\psi}\mp\cos\psi)^2}=\sin\psi\mid \sqrt{\tau^2-\sin^2\psi}\mp\cos\psi\mid
\end{split}
\end{equation}
 We obtain, by plugging the relation (\ref{thetasintheta}) into  (\ref{deripsi}): 

  \begin{equation}\begin{split}\label{deripsib}
\frac{d\theta}{d\psi}&=-\frac{1}{ \mid \sqrt{\tau^2-\sin^2\psi}\mp\cos\psi\mid\sqrt{\tau^2-\sin^2\psi}}\left(\cos\psi\mp\sqrt{\tau^2-\sin^2\psi}\right)^2\\
&=-\mid 1\mp\frac{\cos\psi}{\sqrt{\tau^2-\sin^2\psi}}\mid
\end{split}
\end{equation}

Hence, 
\begin{equation}\begin{split}
& \int_{0}^{\pi}\frac{\sin^{n-2}\theta}{(1-2\tau\cos\theta+\tau^2)^{\frac{n+2s}{2}}}d\theta\\
 &= \int_{0}^{\pi}(\frac{\sin\psi}{\tau})^{n+2s} \sin^{-2-2s}\theta \left(1\mp\frac{\cos\psi}{\sqrt{\tau^2-\sin^2\psi}}\right)d\psi\\
 &= \tau^{2+2s}\int_{0}^{\pi}(\frac{\sin\psi}{\tau})^{n+2s} \left( \sin\psi\mid \sqrt{\tau^2-\sin^2\psi}\mp\cos\psi\mid\right)^{-2-2s}\left(1\mp\frac{\cos\psi}{\sqrt{\tau^2-\sin^2\psi}}\right)d\psi\\
&=\tau^{2-n}\int_{0}^{\pi} \sin^{n-2}\psi\left( \mid \sqrt{\tau^2-\sin^2\psi}\mp\cos\psi\mid\right)^{-2-2s}\left(1\mp\frac{\cos\psi}{\sqrt{\tau^2-\sin^2\psi}}\right)d\psi\\
&=\tau^{2-n}\int_{0}^{\pi} \sin^{n-2}\psi
\frac{
1 }
{ (
 \sqrt{
 \tau^2-\sin^2\psi
 }\mp\cos\psi
  )^{1+2s} \sqrt{\tau^2-\sin^2\psi}}d\psi\\
  &=\tau^{2-n}  ( \tau^2-1
  )^{-1-2s} \int_{0}^{\pi} \sin^{n-2}\psi
\frac{  (
 \sqrt{
 \tau^2-\sin^2\psi
 }+\cos\psi
  )^{1+2s}
}
{\sqrt{\tau^2-\sin^2\psi}}d\psi. 
\end{split}
\end{equation}
Let us denote
\begin{equation}\label{h-tau}
H(\tau)=2\pi\alpha_n \int_{0}^{\pi} \sin^{n-2}\psi
\frac{  (
 \sqrt{
 \tau^2-\sin^2\psi
 }+\cos\psi
  )^{1+2s}
}
{\sqrt{\tau^2-\sin^2\psi}}d\psi.
\end{equation}
Then $H(\tau) = \tau^{n-2}  ( \tau^2-1
  )^{1+2s} \, K(\tau)$. 
Clearly,  $H(\tau)$ is a positive continuous function on $[1, +\infty)$ such that  
$H(\tau)\simeq \tau^{2s}$ as $\tau \to +\infty$. Moreover, it follows from 
(\ref{k-tau}) that, as was mentioned in the Introduction, $H(\tau)$ can be expressed in terms of the hypergeometric function $_{2}F_1(a,b,c, \tau)$ (see \cite{MOS}, p. 55): 
$$
 H(\tau) =  \tau^{-2-2s} ( \tau^2-1
  )^{1+2s} \frac{2 \pi^{\frac{n}{2}}} {\Gamma(\frac{n}{2})} \, 
 _{2}F_1(a, b, c, \tau^{-2}),  \quad \tau \ge 1,
$$
where
$$
a=\frac{n+2s}{2}, \quad b= 1+s, \quad c=\frac{n}{2}. 
$$
On the other hand
\begin{equation}\begin{split}
(-\Delta)^su(r)&=\int_0^{+\infty}\frac{u(r)-u(r\tau)}{r^{2s}}\tau^{n-1}K(\tau)d\tau\\
&=\int_0^{1}\frac{u(r)-u(r\tau)}{r^{2s}}\tau^{n-1}K(\tau)d\tau+\int_1^{+\infty}\frac{u(r)-u(r\tau)}{r^{2s}}\tau( \tau^2-1
  )^{-1-2s} H(\tau)d\tau.
\end{split}
\end{equation}
 Let us consider the following integral
 $$
 \int_0^{1}\frac{u(r)-u(r\tau)}{r^{2s}}\tau^{n-1}K(\tau)d\tau,
 $$
 and perform the following change of variable: $\xi=\frac{1}{\tau}.$ Then
 \begin{equation}\begin{split}
  \int_0^{1}\frac{u(r)-u(r\tau)}{r^{2s}}\tau^{n-1}K(\tau)d\tau&= \int_1^{+\infty}\frac{u(r)-u(\frac{r}{\xi})}{r^{2s}}\xi^{-n+1}\xi^{-2}K(\frac{1}{\xi})d\xi\\
  &=\int_1^{+\infty}\frac{u(r)-u(\frac{r}{\xi})}{r^{2s}}\xi^{-n-1} K(\frac{1}{\xi})d\xi. 
\end{split}
\end{equation}
Notice that
  $$
 K(\frac{1}{\xi})=2\pi \alpha_n \xi^{n+2s}\int_{0}^{\pi}\frac{\sin^{n-2}\theta}{(\xi^{2}-2\xi\cos\theta+1)^{\frac{n+2s}{2}}}d\theta. 
$$
Hence, 
 \begin{equation}\begin{split}
&  \int_0^{1}\frac{u(r)-u(r\tau)}{r^{2s}}\tau^{n-1}K(\tau)d\tau= \int_1^{+\infty}\frac{u(r)-u(\frac{r}{\xi})}{r^{2s}}\xi^{-n+1}\xi^{-2}K(\frac{1}{\xi})d\xi\\
  &=\int_1^{+\infty}\frac{u(r)-u(\frac{r}{\xi})}{r^{2s}}\xi^{-n-1} \xi^{n+2s}K(\xi)d\xi
=  \int_1^{+\infty}\frac{u(r)-u(\frac{r}{\xi})}{r^{2s}}\xi^{-1+2s}\xi^{2-n}( \xi^2-1
  )^{-1-2s} H(\xi)d\xi\\
  &=\int_1^{+\infty}\frac{u(r)-u(\frac{r}{\xi})}{r^{2s}}\xi^{1-n+2s}( \xi^2-1
  )^{-1-2s} H(\xi)d\xi.
\end{split}
\end{equation}
 Thus
 \begin{equation}\begin{split}\label{rapres}
(-\Delta)^su(r)=r^{-2s}\int_1^{+\infty}\left(u(r)-u(r\xi)+(u(r)-u(\frac{r}{\xi}))\xi^{-n+2s}\right)\xi( \xi^2-1
  )^{-1-2s} H(\xi)d\xi.
\end{split}
\end{equation}
The convergence of the integral in (\ref{rapres}) is discussed 
below in Sec.~\ref{appendixx}. 
\end{proof} 

\bigskip

We now deduce Theorem~\ref{charachter_s_subharmonic_int}  
from  Theorem~\ref{radialrappresentation} using mollification 
defined by means of Mellin's convolution which preserves condition (\ref{cond-1}). 

\begin{proof}[Proof of Theorem~\ref{charachter_s_subharmonic_int}]  
Suppose $u$ is a radially symmetric upper semicontinuous function satisfying the conditions  
\begin{equation}\label{cond-int}
\int_0^{+\infty}\frac{\mid u(r) \mid}{(1+r)^{n+2s}} \,  r^{n-1} \, dr < +\infty,
\end{equation}
and 
\begin{equation}\label{pos-cond}
u(r)-u(r\xi)+(u(r)-u(\frac{r}{\xi}))\xi^{-n+2s} \le 0, \quad \text{for all} \, \,  r>0, \, \, \xi \ge 1. 
\end{equation} 
Let us show that $u$ is $s$-subharmonic, i.e., $-(-\Delta)^{s} u\ge 0$ in the sense of distributions. Let 
$\phi \in C^\infty_0 (\mathbb{R})$ so that $\phi \ge 0$, $\phi$ is supported in the interval $|r|\le r_0$, where $0<r_0<+\infty$, and $\int_{\mathbb{R}} \phi(y) dy =1$. 
For $\epsilon>0$, define the approximate identity on $\mathbb{R}_{+}$ by:  
\begin{equation}\label{approx-id}
\phi_\epsilon(r) = \frac{1}{\epsilon} \phi\left(\frac{\log r}{\epsilon}\right), 
\quad r>0. 
\end{equation} 
Then clearly, for every $\epsilon>0$, 
\begin{equation} \label{prob-meas}
\int_0^{+\infty} \phi_\epsilon(r) \frac{dr}{r}=1. 
\end{equation} 

We observe that  (\ref{pos-cond}) is invariant under the Mellin convolution:
\begin{equation}\label{mollific}
u_\epsilon(\tau) =\int_{0}^{+\infty}  \phi_{\epsilon} \left(\frac{\tau}{t}\right) \, u(t) \frac{dt}{t}, 
\quad \tau>0. 
\end{equation} 
Indeed, integrating both sides of (\ref{pos-cond}) against 
$ \phi_{\epsilon} (\frac{\tau}{r}) \frac{dr}{r}$, we obtain:  
\begin{equation}\label{pos-cond1}
u_{\epsilon} (\tau)-u_{\epsilon}(\tau\xi)+(u_{\epsilon}(\tau)-u_{\epsilon}(\frac{\tau}{\xi}))\xi^{-n+2s} \le 0, \quad \text{for all} \, \,  \tau>0, \, \, \xi \ge 1. 
\end{equation} 
Moreover, obviously $u_{\epsilon}\in C^\infty(\mathbb{R}_{+})$, and 
\begin{equation}\label{cond-int1}
\int_0^{+\infty}\frac{|u_{\epsilon}(r)| \, r^{n-1}}{(1+r)^{n+2s}} dr 
\le C \int_0^{+\infty}\frac{|u(r)| \, r^{n-1}}{(1+r)^{n+2s}} dr  < +\infty, 
\end{equation}
where $C=C(\epsilon, r_0, s, n)$ is a positive constant. 
Indeed, by Fubini's theorem, 
\begin{equation*}
\int_0^{+\infty}\frac{|u_{\epsilon}(r)| \, r^{n-1}}{(1+r)^{n+2s}} dr 
\le  \int_0^{+\infty} |u(t)| \int_0^{+\infty} \frac{r^{n}}{(1+r)^{n+2s}} 
\phi_\epsilon\big( \frac{r}{t}\big)  \frac{dr}{r} \,  \frac{dt}{t}
\end{equation*}
\begin{equation*}
= \int_0^{+\infty} |u(t)| \int_0^{+\infty} \frac{\lambda^n t^{n}}{(1+ \lambda t)^{n+2s}} 
\phi_\epsilon\big(\lambda) \frac{d\lambda}{\lambda} \, \frac{dt}{t} 
\le C \int_0^{+\infty}\frac{|u(t)| \, t^{n}}{(1+t)^{n+2s}} \frac{dt}{t}  < +\infty, 
\end{equation*}
where the last estimate follows since $\phi_\epsilon(\lambda)$ vanishes outside the interval $( e^{-\epsilon r_0}, \, e^{\epsilon r_0} )$, and 
(\ref{prob-meas}) holds.

Let $h \in C^\infty_0(\mathbb{R}^n\setminus\{0\})$ be a nonnegative test function supported in  $0< |x|\le R<+\infty$.  
We observe that, since $h$ is compactly supported, it follows from (\ref{lap-def}) that $(-\Delta)^{s} h \in C^\infty(\mathbb{R}^n)$, and 
\begin{equation} \label{est-h}
 \vert (-\Delta)^{s} h(x) \vert  \le \frac{C}{(1+|x|)^{n+2s}},  
  \quad x \in \mathbb{R}^n.
\end{equation} 
In particular, 
\begin{equation*}
\langle (-\Delta)^{s} u, \, h\rangle =  \int_{\mathbb{R}^n} u \, (-\Delta)^{s} \, h\, dx 
\end{equation*}
is well-defined in terms of distributions (see \cite{Lan}, Sec. 1.6). By Theorem~\ref{radialrappresentation} and (\ref{pos-cond1}), $ (-\Delta)^{s} u_\epsilon \ge 0$, and hence,  using (\ref{cond-int}) 
and recalling that $u_\epsilon \in C^\infty(\mathbb{R}^n\setminus\{0\})$, we obtain: 
\begin{equation}\label{est-h1}
\langle (-\Delta)^{s} u_\epsilon, \, h\rangle = \langle u_\epsilon, \, (-\Delta)^{s} h\rangle = \int_{\mathbb{R}^n} u_\epsilon  \, (-\Delta)^{s} h \, dx \ge 0,
\end{equation} 
for every $h \in C^\infty_0(\mathbb{R}^n\setminus\{0\})$, $h \ge 0$. 

It remains to prove the approximation property: 
\begin{equation}\label{approx-prop}
\lim_{\epsilon \to 0}   \int_{\mathbb{R}^n} u_\epsilon \,  (-\Delta)^{s} \, h \, dx =   \int_{\mathbb{R}^n} u \, (-\Delta)^{s} h\, dx. 
\end{equation} 
Denote by $\psi(t)$ the spherical mean of $(-\Delta)^{s} h$:
\begin{equation*}
\psi (t) = \int_{|x'|=1}  (-\Delta)^{s} h (tx')\, d H^{n-1}(x'), \quad t>0.  
\end{equation*} 
Then 
\begin{equation*}
 \int_{\mathbb{R}^n} [u(x) - u_\epsilon(x)] \, (-\Delta)^{s} h (x)\, dx = \int_0^{+\infty}[ u(t)- u_\epsilon(t)] \, \psi(t) \, t^n \frac{dt}{t}.
 \end{equation*}
 We  estimate:
 \begin{equation*}
\left \vert  \int_0^{+\infty} [ u(t)-u_\epsilon (t)] \psi(t) t^n  \frac{dt}{t} 
\right \vert = \left \vert  \int_0^{+\infty} u(t) \left [  \psi (t) t^n - 
 \int_0^{+\infty} \phi_\epsilon \big ( \frac {r}{t} \big )  \psi(r)  r^n
  \frac{dr}{r} \right ] \,  \frac{dt}{t} \right \vert 
  \end{equation*}
   \begin{equation*}
   \le  \int_0^{+\infty}  |u(t)| \,  \int_0^{+\infty} 
   | t^{n} \psi(t) - r^n \psi (r)| \, \phi_\epsilon \big ( \frac {r}{t} \big ) 
    \frac{dr}{r} \, \frac{dt}{t}.
    \end{equation*} 
    We recall that, since $\phi(y)$ is supported in the 
      interval $|y|\le r_0$, it follows that $\phi_\epsilon \big ( \frac {r}{t} \big ) 
$ is supported in the interval where $\mid \log \frac r t \mid \le \epsilon r_0$. Hence, in the above estimates we can assume: 
   \begin{equation*}
  e^{-\epsilon r_0}-1  \le  \frac r t -1 \le e^{\epsilon r_0}-1. 
      \end{equation*} 
      From this  we deduce:
      \begin{equation}\label{eps-delta}
  |r -t| \le \delta t,
      \end{equation} 
      where $\delta = e^{\epsilon r_0}-1\to 0$ as $\epsilon \to 0$.

    We notice that, since $(-\Delta)^{s} h \in C^\infty(\mathbb{R}^n)$, 
it follows that its spherical mean  $\psi$ is infinitely 
differentiable on $[0, +\infty)$ (see, e.g., \cite{Lan}, 
Sec. I.6). 
    We will need the following estimates:  
       \begin{equation}\label{est0-1} 
    |\psi(t)| \le \frac{C} {(1+t)^{n+2s}}, \quad  |\psi'(t)| \le 
     \frac{C} {(1+t)^{n+2s+1}}, \quad t\ge 0,  
      \end{equation}
      where $C$ depends on $h$. Indeed,  $h(x)=0$ for $|x|>R$, and  
      consequently  (\ref{lap-def}) yields, for $t>R$:  
       \begin{equation*}
    \psi(t)= -c_{s,n}  \int_{|y|\le R} h(y)   \int_{|x'|=1} \frac{1}{|y-t x'|^{n+2s}} d H^{n-1} (x')\, dy, 
      \end{equation*} 
       \begin{equation*}
    \psi'(t)= c_{s,n} (n+2s) \, \int_{|y| \le R} 
    h(y)  \int_{|x'|=1} \frac{t-x' \cdot y}{|y-t x'|^{n+2s+2}} d H^{n-1} (x')\, dy  
      \end{equation*} 
      \begin{equation*}
 = c_{s,n} (n+2s) \, t \,  \int_{|y|\le R} h(y)  \int_{|x'|=1} 
    \frac{1}{|y-t x'|^{n+2s+2}} d H^{n-1} (x')\, dy. 
          \end{equation*} 
          Here  $|y-t x'| \ge t-R > \tfrac t 2$ for $t>2R$, which yields (\ref{est0-1}) for $t>2R$.  
      Since $\psi$ is infinitely differentiable on $[0, +\infty)$, it follows  that  (\ref{est0-1}) holds for all $t\ge 0$.

    Let $\psi_1(t) = t^n \psi(t)$. Clearly, estimates (\ref{est0-1}) 
    yield: 
     \begin{equation}\label{est0-2}
     |\psi_1(t)| \le 
     \frac{C \, t^{n}} {(1+t)^{n+2s}}, \quad  |\psi_1'(t)| \le 
     \frac{C \, t^{n-1}} {(1+t)^{n+2s}}, \quad t\ge 0. 
      \end{equation} 
     Invoking  the mean value inequality we estimate: 
 \begin{equation*}
            | t^{n} \psi(t) - r^n \psi (r)| =  | \psi_1(t) - \psi_1 (r)| \le 
            \, |r-t| \, |\psi_1'(\lambda)| \le C \, |r-t| \,   \frac{ \lambda^{n-1}} {(1+\lambda)^{n+2s}}, 
              \end{equation*}
                        for some $\lambda$ between $t$ and $r$. Assuming by (\ref{eps-delta}) that  $|r-t|< \delta t$, we see that $t(1-\delta) \le \lambda \le t(1+\delta)$. Combining 
          the preceding estimates, we obtain:
           \begin{equation*}
            | t^{n} \psi(t) - r^n \psi (r)| \le \frac {C \,  \delta \,  (1+\delta)^{n-1}}{ (1-\delta)^{n+2s}}  \, \frac{t^{n}} { (1+t)^{n+2s}}. 
             \end{equation*} 
            Using this together with (\ref{prob-meas}), we 
            conclude:  
    \begin{equation*}
    \int_0^{+\infty}  |u(t)| \,  \int_0^{+\infty} 
   | t^{n} \psi(t) - r^n \psi (r)| \, \phi_\epsilon \big ( \frac {r}{t} \big ) 
    \frac{dr}{r} \, \frac{dt}{t} \le  \frac {C \,  \delta \,  (1+\delta)^{n-1}}{ (1-\delta)^{n+2s}}
     \int_0^{+\infty}  |u(t)| \,  \, \frac{t^{n-1}} { (1+t)^{n+2s}}dt, 
    \end{equation*} 
where the right-hand side is finite by (\ref{cond-int}). Letting $\epsilon \to 0$, and hence $\delta \to 0$, we conclude 
      the proof of the approximation property (\ref{approx-prop}).                
    \end{proof}

\section{Fractional subharmonicity for radially symmetric functions}\label{appendixx}

In this section we study further the condition 
\begin{equation}\label{positive}
u(r)-u(r\xi)+(u(r)-u(\frac{r}{\xi}))\xi^{-n+2s} \le 0, \quad \text{for all} \, \,  r>0, \, \, \xi \ge 1,
\end{equation}
which ensures that a radially symmetric $C^2$ function $u$ is $s$-subharmonic, i.e., $-(-\Delta)^s u\ge 0$, for $0<s\le 1$. However, this condition makes sense for any real $s$. We will show that 
for $C^2$ functions $u$ it is equivalent to:
\begin{equation}\label{positive1}
u''(r) +(n+1-2s)\frac{u'(r)}{r}\ge 0,  \quad \text{for all} \, \, \,  r>0,
\end{equation} 
which is stated as Theorem~\ref{s-subharm} in the Introduction.

Suppose $u$ is a radial $C^2$ function. It is worth noting that in a neighborhood of $1,$ whenever $\xi\to 1,$ $\xi\geq 1,$
\begin{equation*}\begin{split}
& u(r)-u(r\xi)+(u(r)-u(\frac{r}{\xi}))\xi^{-n+2s}=-u'(r)r\frac{(\xi-1)^2(\xi^{n+1-2s}-1)}{\xi^{n+1-2s}(\xi-1)}-r^2\frac{u''(r)}{2}(\xi-1)^2(1+\frac{1}{\xi^{n+2-2s}})\\
&+\xi^{-n+2s}o((\xi-1)^2)
 =-\frac{1}{2}r^2\frac{(\xi-1)^2(\xi^{n+2-2s}+1)}{\xi^{n+2-2s}}(u''(r)+\frac{2\xi(\xi^{n+1-2s}-1)}{r(1+\xi^{n+2-2s})(\xi-1)}u'(r))+o(\xi-1)^2\\
 &=-\frac{1}{2}r^2\frac{(\xi-1)^2(\xi^{n+2-2s}+1)}{\xi^{n+2-2s}}\mathcal{L}u,
\end{split}
\end{equation*}
where
\begin{equation*}\begin{split}
&\mathcal{L}u=u''+\frac{2\xi(n+1-2s+\frac{n-2s}{2}(\xi-1)+\frac{n-1-2s}{6}(\xi-1)^2+o(\xi-1)^2)}{(1+\xi^{n+2-2s})}\frac{u'}{r}+o(r,(\xi-1)^2)\\
&=u''+(\frac{2\xi(n+1-2s)}{1+\xi^{n+2-2s}}+g(\xi-1))\frac{u'}{r}\\
&=u''+(n+1-2s)\frac{u'}{r}\\
&+\frac{(n+1-2s)(2\xi-1-\xi^{n+2-2s})+\frac{n-2s}{2}(\xi-1)+\frac{n-1-2s}{6}(\xi-1)^2\frac{u'}{r}
 +o(\xi-1)^2)}{(1+\xi^{n+2-2s})}\\&+o(r,(\xi-1)^2)=u''+(n+1-2s)(1+h(\xi-1))\frac{u'}{r},
\end{split}
\end{equation*}
where $\lim_{\xi\to 1^{+}}\frac{h(\xi-1)}{\xi-1}=1.$

In particular, the integral in (\ref{rapresth}) is convergent in a neighborhood of $1,$ since for any fixed $r>0$ it behaves as $(1-\xi)^{1-2s}$ that converges whenever  $0<s<1.$ Moreover, by (\ref{cond-int}) it converges when $\xi\to +\infty$.  Hence the integral in (\ref{rapresth}) is convergent for every $C^2$ function 
$u$ satisfying (\ref{cond-int}). 

The above calculation 
demonstrates that condition (\ref{positive}) is closely related to 
condition (\ref{positive1}) for any $s\in \mathbb{R}$, which is proved below.

\begin{proof}[Proof of Theorem \ref{s-subharm}] Let $s\in \mathbb{R}$. Suppose $r>0$ and $\xi \ge 1$. Dividing both sides of 
(\ref{positive}) by $(\xi-1)^2$ and passing to the limit as $\xi \to 1^{+}$, we deduce (\ref{positive1}). 

Conversely, suppose that (\ref{positive1}) 
holds, and $u'(r) \ge 0$. Let $\gamma=n-2s$. Making a substitution $r=e^t$, $\xi=e^x$, 
where $t\in \mathbb{R}$, and $x \ge 0$, we rewrite (\ref{positive}) 
in the equivalent form:
\begin{equation}\label{positive2}
(v(t+x)-v(t)) e^{\gamma x} +v(t-x)-v(t) \ge 0, \quad \text{for all} \, \, t \in \mathbb{R}, \, \, x \ge 0,
\end{equation}
where $v(t)=u(e^t)$, while (\ref{positive1}) is equivalent to 
\begin{equation}\label{positive5}
v''(t) + \gamma v' (t) \ge 0, \quad \text{for all} \, \, t \in \mathbb{R}.
\end{equation}

We next fix $t \in \mathbb{R}$, and let $\phi(x) = v(x+t) -v(t)$, where $x \in \mathbb{R} $. 
Then clearly $\phi(0)=0$, $\phi'(0)=v'(t)$, and  (\ref{positive5}) is equivalent to 
\begin{equation}\label{positive6}
\phi''(x) + \gamma \phi'(x) \ge 0, \quad \text{for all} \, \,  x \in \mathbb{R}. 
\end{equation}

 We need to prove:
\begin{equation}\label{est}
\phi(x) e^{\gamma x} +\phi(-x) \ge 0, \quad x \ge 0,
\end{equation} 
which is equivalent to (\ref{positive2}). It is not difficult to deduce (\ref{est})  from (\ref{positive6}) using  Gronwall's
 inequality (see \cite{Har}, Sec. 3.1, Theorem 1.1) in the case 
 $\gamma>0$. 
 
Let us prove this directly for all $\gamma \in \mathbb{R}$. By   (\ref{positive6}), it folows that $\phi'(x) + \gamma \phi(x)$ is non-decreasing on $\mathbb{R}$. Since $\phi(0)=0$, we obtain:  
\begin{equation}\label{est1}
\phi'(x) + \gamma \phi(x) \ge \phi'(0), \quad \text{for all} \, \, x \ge 0.
\end{equation}

Let 
\begin{equation*}
\psi(x)= \phi(-x) + \phi(x) e^{\gamma x}, \quad x \in \mathbb{R}.  
\end{equation*} 
We next show that $\psi' (x) \ge 0$ for all $x \ge 0$. By 
(\ref{est1}), 
\begin{equation*}
\psi'(x)=-\phi'(-x)+  \left (\phi'(x) + \gamma \phi(x) \right )e^{\gamma x} \ge -\phi'(-x)+ \phi'(0) e^{\gamma x}, \quad \text{for all} \, \, x \ge 0.
\end{equation*} 
It remains to prove the inequality  
\begin{equation}\label{est2}
\phi'(-x) \le  \phi'(0) \, e^{\gamma x}, \quad \text{for all} \, \, x \ge 0.  
\end{equation} 

Let $g(x)=e^{-\gamma x} \phi'(-x) $. Then 
\begin{equation*}
g'(x)=-\gamma e^{-\gamma x} \phi'(-x) - e^{-\gamma x} \phi''(-x) 
=-e^{-\gamma x} ( \phi''(-x) + \gamma  \phi'(-x) )\le 0, 
\end{equation*} 
by (\ref{positive6}). Hence $g$ is nonincreasing, and consequently $g(x) \le g(0)$ for all $x \ge 0$. This proves 
(\ref{est2}). Thus $\psi' (x)\ge 0$ for all $x \ge 0$, and since $\psi(0)=0$,  we deduce 
$\psi (x) \ge 0$  for all $x \ge 0$. This proves (\ref{est}), which in its turn yields (\ref{positive}). 
  
\end{proof}

\section{Examples of $s-$subharmonic functions and further remarks}\label{examples_remarks}

We consider the fundamental solution of the $k$-Hessian operator $F_k$, 
$$
u(x)=-\mid x\mid^{-(\frac{n}{k}-2)}, \quad x \in \mathbb{R}^n,
$$
for $1\le k \le \frac n 2$, which satisfies the equation
$$
F_k[u]=\delta_0,
$$
in the viscosity sense (see \cite{TW2}, \cite{W2}). 

It is not difficult to verify directly  that $u$ is a radial function satisfying 
the condition 
\begin{equation}\label{positive3}
u(r)-u(r\xi) + (u(r)-u(\frac{r}{\xi}))\xi^{-n+2s} \le 0, \quad \text{for all} \, \,  r>0, \, \, \xi \ge 1,
\end{equation}
if $s=\frac{n(k-1)}{2k}+1$, or equivalently $n-2s= \frac{n}{k}-2$.

More generally, for $\beta>0$, let us consider the function 
 $$
u(x)=-\mid x\mid^{-\beta}, \quad x \in \mathbb{R}^n. 
$$
If $\beta=n-2s$, we have:
$$
\frac{u(r)-u(r\xi)}{u(\frac{r}{\xi})-u(r)}=
\frac{-r^{-(n-2s)}+r^{-(n-2s)}\xi^{-(n-2s)}}{-r^{-(n-2s)}\xi^{n-2s}+r^{-(n-2s)}}
$$
$$=\frac{-1+\xi^{-(n-2s)}}{-\xi^{n-2s}+1}=\xi^{-(n-2s)}.
$$ 
It follows that, for $\beta=n-2s$ and  $s<\frac n 2$, condition (\ref{positive3}) turns into an equality for all $r>0, \,  \xi \ge 1$. Unfortunately the $k$-energy of $u(x)=-\mid x\mid^{-(\frac{n}{k}-2)}$ is unbounded, and 
so such functions cannot be used in Hessian inequalities of the type (\ref{eq-2.int}) or (\ref{hessian-sobolev}) 
discussed above. 

This suggests considering the function 
$$
f_\beta (x)=-(1+\mid x\mid^2)^{-\frac{\beta}{2}}, \quad x \in \mathbb{R}^n,
$$ 
for $\beta>0$. When $\beta= \frac{n}{k}-2$, $f$ is a $k$-convex function of finite $k$-energy. Moreover, $f_\beta$ is known to be  an extremal  function for the Hessian Sobolev inequality (\ref{hessian-sobolev}) (see \cite{W1}, \cite{W2}). 

\bigskip

\begin{corollary}\label{quasifunda_s_subh_int}   Let $s\in \mathbb{R}$ and $\beta>0$. 
Then $f_\beta$ satisfies  (\ref{cond-1}), and consequently   (\ref{cond-2}), if and only if $\beta\le n- 2s$. Moreover, if $0<s\le1$ and $0<\beta<n$, then $f_\beta$ is 
$s$-subharmonic in $\mathbb{R}^n$ if and only 
if $\beta\le n- 2s$. 
\end{corollary}

\begin{proof} 
Let $r=|x|$. It is easy to see that
$$
f'_\beta (r) = -\beta r (r^2+1)^{-\frac{\beta}{2}-1}, \quad 
f''_\beta (r) = \beta (r^2+1)^{-\frac{\beta}{2}-2}\left [ (\beta+1) r^2-1 \right]. 
$$
Hence,
$$
f''_\beta (r) + \frac{n-2s+1}{r} f'_\beta (r)  = \beta (r^2+1)^{-\frac{\beta}{2}-2} 
\left [ (\beta-n+2s) r^2 -(n-2s+2)\right]. 
$$
Suppose $0<\beta \le n-2s$. Then the right-hand side of the preceding inequality is negative, and  (\ref{cond-2}) holds. 
If $0<s\le 1$, this implies that $f_\beta$ is $s$-subharmonic. 

If $\beta>n-2s$, then $f''_\beta (r) + \frac{n-2s+1}{r} f'_\beta (r)$ is 
either positive or changes sign. Hence,  (\ref{cond-2}) fails in this case.

Moreover,  by formula 
(\ref{explicit}) discussed below, 
$$
\phi(r)=-(-\Delta)^{s} f_\beta (r)= C(\alpha, \beta, n) \, F(a,b,c, -r^2),
$$
where $C(\alpha, \beta, n)$ is a positive constant, and 
 \begin{equation}\label{parameters-1}  
a= \frac{n+2s}{2}, \quad b= \frac{2s+\beta}{2}, \quad c= \frac{n}{2}, 
   \end{equation}
where $F(x) = _2\!\!\!F_1 (a, b, c, x)$ is the hypergeometric function. Notice that $F(0)=1$ and by a formula due to Gauss 
 (see \cite{AAR}, Theorem 2.2.2, p. 66): 
\begin{equation}\label{parameters-2}  
F(1)=\frac{\Gamma(c) \Gamma(c-a-b)}{\Gamma(c-a) \Gamma(c-b)}, 
\quad \text{if} \, \, \, c>a+b.
   \end{equation}
By Pfaff's transformation,
$$
F(a,b,c,-r^2)=(1+r^2)^{-b} F(c-a,b,c, \frac{r^2}{r^2+1}). 
$$
It follows:  
\begin{equation}\label{f-asymp}
F(-r^2) \sim (1+r^2)^{-b} \frac{\Gamma(\frac{n}{2}) \Gamma(\frac{n-\beta}{2})}{\Gamma(\frac{n+s}{2}) \Gamma(\frac{n-\beta}{2}-s)} 
\quad \text{as} \, \, \, r \to +\infty.
\end{equation}
Hence, when $n-2s<\beta<n$, $0<s<1$, $\phi(r)<0$ for $r$ large, 
and so changes sign. In other words, $f_\beta$ fails to be $s$-subharmonic in this case. 
\end{proof}

\section{Applications}\label{applicationsxx}

We first recall the following definitions. Let $s\in (0, 1]$. Let $u$ be any continuous (or more generally upper semicontinuous) function   $u: \, \mathbb{R}^n \to \mathbb{R}\cup \{-\infty\}$ 
such that 
$$\int_{\mathbb{R}^n }
\frac{| u(x) |}
{(1+ | x |^2)^{\frac{n+2s}{2}}}
dx < +\infty.
$$
We say that $u$ is $s$-subharmonic in $\mathbb{R}^n$  if 
$$
-(-\Delta)^s u\ge 0 \quad \text{in} \, \, D'(\mathbb{R}^n).   
$$
Analogously we shall say that $u$ is $s$-superharmonic 
in $\mathbb{R}^n$ 
$$
 -(-\Delta)^s u\le 0 \quad \text{in} \, \, D'(\mathbb{R}^n).   
$$
Whenever $u$ is both $s$-subharmonic and $s$-superharmonic in 
 we shall say that $u$ is $s$-harmonic
in $\mathbb{R}^n$.

We remark that from the representation given in Theorem \ref{radialrappresentation} we can deduce 
a Liouville theorem for radial $s$-subharmonic  functions.
\begin{corollary}
Assume that $u$ is a positive continuous function, radially decreasing and $s\in (0,1].$ Suppose that
$$
\int_{\mathbb{R}^n}\frac{\mid u(x)\mid}{(1+\mid x\mid^2)^{\frac{n+2s}{2}}}dx<+\infty,
$$
and $-(-\Delta)^s u\geq 0$. If
$\lim_{r\to+\infty}u(r)=0$ then $u=0.$
\end{corollary}
\begin{proof} We remark that if $u$ is radially decreasing then $u(r)\geq u(r\xi)$ and $u(r)\leq u(\frac{r}{\xi})$ for every $\xi\geq 1.$ Hence, recalling Theorem \ref{charachter_s_subharmonic_int} inequality \ref{cond-1},  for every $\xi\geq 1$ and for 
every positive $r$ it results
$$
\frac{u(r)-u(\frac{r}{\xi})}{u(r\xi)-u(r)}\geq \xi^{n-2s}.
$$
On the other hand, for every fixed $r,$ if $\xi\to +\infty,$ then
$$
\lim_{\xi\to+\infty}\frac{u(r)-u(\frac{r}{\xi})}{u(r\xi)-u(r)}=+\infty.
$$
We recall that $\lim_{r\to+\infty}u(r)=0,$ so  we get 
$$
\lim_{\xi\to+\infty}\frac{u(r)-u(0)}{-u(r)}=+\infty, 
$$
which implies a contradiction whenever $u$ is bounded and not identically zero.
\end{proof}

The following result gives a derivative rule for fractional Laplacians.
\begin{theorem}
Let $s\in [0,1].$ For every differentiable  radial  function  $u$  such that
$$
\int_{\mathbb{R}^n}\frac{\mid u(x)\mid}{(1+\mid x\mid^2)^{\frac{n+2s}{2}}}dx<+\infty,
$$
 the following formula holds: 
  \begin{equation}\begin{split}\label{rapresthderd}
\frac{d(-\Delta)^su(r)}{dr}  =\frac{1}{r}(-\Delta)^s(-2su+ru').
  \end{split}
\end{equation}
\end{theorem}
  \begin{proof} Differentiating (\ref{rapresth}), we obtain: 
 \begin{equation*}\begin{split}\label{rapresthderdxy}
&\frac{d(-\Delta)^su(r)}{dr}=-\frac{2s}{r}(-\Delta)^su(r)\\
&+r^{-1-2s}\int_1^{+\infty}\left((u'(r)-u'(r\xi)\xi+(u'(r)-\frac{u'(\frac{r}{\xi})}{\xi})\xi^{-n+2s}\right)\xi( \xi^2-1
  )^{-1-2s} H(\xi)d\xi\\
  &=-\frac{2s}{r}(-\Delta)^su(r)\\
&+r^{-2s}\int_1^{+\infty}\left(r^{-1}\left((ru'(r)-u'(r\xi)r\xi)+(ru'(r)-\frac{u'(\frac{r}{\xi})}{\frac{\xi}{r}})\right)\xi^{-n+2s}\right)\xi( \xi^2-1
  )^{-1-2s} H(\xi)d\xi\\
  &=-\frac{2s}{r}(-\Delta)^su(r)+\frac{1}{r}(-\Delta)^s(ru'(r))=\frac{1}{r}(-\Delta)^s(-2su+ru'),
\end{split}
\end{equation*}
\end{proof}

If
$$
f_k(x)=-(1+r^2)^{-(\frac{n}{2k}-1)},
$$
then
$$
f'_k(r)=2(\frac{n}{2k}-1)(1+r^2)^{-\frac{n}{2k}}r
$$
and
recalling formula (\ref{rapresthderd})
we get
\begin{equation*}
\begin{split}
&-2sf_k(r)+rf'_k(r)=2s(1+r^2)^{-(\frac{n}{2k}-1)}+2(\frac{n}{2k}-1)(1+r^2)^{-\frac{n}{2k}}r^2\\
=&2(1+r^2)^{-\frac{n}{2k}}\left(s(1+r^2)+(\frac{n}{2k}-1)r^2\right)=2(1+r^2)^{-\frac{n}{2k}}\left(s+(s+\frac{n}{2k}-1)r^2\right).
\end{split}
\end{equation*}
Thus
\begin{equation*}\begin{split}\label{rapresthderc}
&\frac{d(-\Delta)^sf_k(r)}{dr}=\frac{1}{r}(-\Delta)^s(-2sf_k+rf_k')=\frac{2}{r}(-\Delta)^s\left((1+r^2)^{-\frac{n}{2k}}\left(s+(s+\frac{n}{2k}-1)r^2\right)\right)\\
&=\frac{2}{r}(-\Delta)^s\left((1+r^2)^{-\frac{n}{2k}}\left((s+\frac{n}{2k}-1)+(s+\frac{n}{2k}-1)r^2\right)\right)\\
&+\frac{2}{r}(-\Delta)^s\left((1+r^2)^{-\frac{n}{2k}}\left(s-(s+\frac{n}{2k}-1)\right)\right)\\
&=\frac{2(s+\frac{n}{2k}-1)}{r}(-\Delta)^s\left((1+r^2)^{-(\frac{n}{2k}-1)}\right)-\frac{2(\frac{n}{2k}-1))}{r}(-\Delta)^s\left((1+r^2)^{-\frac{n}{2k}}\right).
\end{split}
\end{equation*}

In the next theorem we deduce  the maximum principle for radial $s$-subharmonic functions.
 
\begin{theorem} 
Suppose that $u$ is a radial $s$-subharmonic function $s\in (0,1].$ Then either $u$ is constant or $u$ can not realize a maximum in $\mathbb{R}^n.$
\end{theorem}
\begin{proof}
Indeed, by contradiction, suppose that $u$ realizes an absolute maximum in $\bar{r}\in\mathbb{R}^n.$ Hence for every $\xi\geq 1,$  
$u(\bar{r})-u(\bar{r}\xi)\geq 0,$ $u(\bar{r})-u(\frac{\bar{r}}{\xi})\geq 0$ and as a consequence
$$
u(\bar{r})-u(\bar{r}\xi)+\left(u(\bar{r})-u(\frac{\bar{r}}{\xi})\right)\xi^{2s-n}\geq 0.
$$ 
On the other hand
$$
-(-\Delta)^{s}(\bar{r})\geq 0,
$$
hence for every $\xi\geq 0$
$$
u(\bar{r})-u(\bar{r}\xi)+\left(u(\bar{r})-u(\frac{\bar{r}}{\xi})\right)\xi^{2s-n}=0
$$ 
and in particular this implies $u\equiv u(\bar{r})$ because $u(\bar{r})-u(\bar{r}\xi)\geq 0$ and $(u(\bar{r})-u(\frac{\bar{r}}{\xi}))\geq 0.$
\end{proof}
\begin{corollary}\label{nominimazfra}
If $u$ is a radial $s$-harmonic function, then either $u$ is constant or $u$ does not realize either an absolute maximum  or an absolute minimum.
\end{corollary}
\begin{corollary}
Suppose $u_1$ and $u_2$ are two radial $s$-harmonic functions. If $u_1\geq u_2,$ then $u_1>u_2,$ or $u_1\equiv u_2$, i.e., $u_1$ cannot intersect with $u_2$; otherwise $u_1$ coincides with $u_2.$
\end{corollary}
\begin{proof} Indeed, if there exists $\bar{r}$ such that $u_1(\bar{r})=u_2(\bar{r}),$ then $u_1-u_2$ has a minimum in $\bar{r}$ and $u_1-u_2$ is $s$-harmonic. Then $u_1\equiv u_2$ by the previous result.
\end{proof}
\begin{corollary}
Suppose that $u$ is a radial continuous $s$-harmonic function vanishing at infinity. Then $u\equiv 0.$ 
\end{corollary}
\begin{proof}
If $\sup_{\mathbb{R}^n}u=0$ then $u$ realizes its minimum in $\mathbb{R}^n$. Then we can apply Corollary \ref{nominimazfra} obtaining that $u$ is constant and necessarely $u\equiv 0.$ Analogously if $\inf_{\mathbb{R}^n}u=0$ then $u$ has to attain  its maximum in $\mathbb{R}^n$ otherwise $u\equiv 0.$ Thus by Corollary \ref{nominimazfra} we conclude that $u\equiv 0$ indeed.
\end{proof}
\begin{corollary}\label{uniqueness}
Let $u_1, u_2$ be  continuous radially symmetric functions such that 
$$
(-\Delta)^su_1=(-\Delta)^su_2.
$$
If $\lim_{r\to+\infty}(u_1-u_2)=0$ then $u_1\equiv u_2.$
\end{corollary}
\begin{proof}
Recalling the linearity of the fractional Laplace operator we get: 
$$
(-\Delta)^s(u_1-u_2)=0. 
$$
Moreover $\lim_{r\to+\infty}(u_1-u_2)=0.$  It follows by Corollary  \ref{uniqueness} that $u_1\equiv u_2.$
\end{proof}

Our results make it possible to provide an answer to a question which was left open in \cite{FFV}. More precisely, in \cite{FFV}, Theorem 2.1, an inequality involving $k$-convex functions and the fractional Laplace operator was proved. For reader's  convenience we state below the result in \cite{FFV}.
\begin{proposition}[Ferrari-Franchi-Verbitsky]\label{direct}
$1 \le k < \frac {n}{2}$, and let $\alpha=\frac{2k}{k+1}$.
Suppose  $u\in C^2(\mathbb{R}^n)$ is a $k-$convex function on  $\mathbb{R}^n$ vanishing at $\infty$. If
\begin{itemize}
\item[(i)] $-(-\Delta)^{\alpha/2}u\geq 0,$
\item[(ii)] $(-\Delta)^{\alpha/2}[-(-\Delta)^{\alpha/2}u]^k\geq 0,$
\end{itemize}
then there exists a  positive constant $C_{k, n}$ such that
\begin{equation}\label{eq-2.1}
\int_{\mathbb{R}^n}\left (-(-\Delta)^{\alpha/2}u\right )^{k+1}dx\leq C_{k, n} \, \int_{\mathbb{R}^n}-u \, F_k[u] \, dx.
\end{equation}
\end{proposition}
It was unclear if, for $1 \le k < \frac {n}{2}$, the set
$$
\mathcal{F}\mathcal{C}_k=\{u\in C^2(\mathbb{R}^n):\:\: u,\:\:k-\mbox{convex},\:\:\mbox{vanishing at }\:\infty,\:\mbox{(i) and (ii) hold}\},
$$
 was not trivial. Obviously  is not empty because for any $k,$  $1 \le k < \frac {n}{2},$ $0\in \mathcal{F}\mathcal{C}_k.$

Theorems~\ref{charachter_s_subharmonic_int}  and \ref{s-subharm} in this paper will be employed to give a positive answer to this question. 
Indeed, 
$$
f_k(x)=-(1+\mid x\mid^2)^{-(\frac{n}{2k}-1)}
$$
is $k$-convex, vanishing at $\infty$, and  $-(-\Delta)^{\alpha/2}f_k\geq 0,$ for every $k\in \mathbb{N}$ such that $\frac{n}{2}\geq k,$ i.e. $f_k$ satisfies (i). Moreover if $k=1,$ then $\alpha=1,$ and condition (ii) becomes
$$
(-\Delta)^{1/2}[-(-\Delta)^{1/2}f_1]=\Delta f_1\geq 0.
$$
The case $k\ge 2$, which  is much harder, is considered 
in the next two sections.

\section{The iterated fractional Laplacian condition}\label{fractional-hessian} 

In this section we verify the iterated fractional Laplacian condition of 
Proposition~\ref{direct}:
   \begin{equation}\label{iterated}
(-\Delta)^{\alpha/2}[-(-\Delta)^{\alpha/2}u]^k\geq 0, 
    \end{equation} 
   where $\alpha=\frac{2k}{k+1}$, $1\le k <  \frac {n}{2}$, for the $k$-convex function $u(x)=-(1+|x|^2)^{-\beta/2}$ on 
    $\mathbb{R}^n$ with 
    $\beta=\frac{n}{k}-2>0$. As explained above, 
    the case $k=1$ is trivial, and from now on we will assume $k\ge 2$. We observe that $u$ is an extremal function in the important Hessian Sobolev inequality of X.-J. 
    Wang \cite{W1}, \cite{W2}.  
    The Fourier transform of $u$ is given by the radially symmetric  function 
     \begin{equation}\label{bessel}  
    \hat u (\xi) = -C(\beta,n) \frac{K_{\frac{n-\beta}{2}}(|\xi|)}{|\xi|^{\frac{n-\beta}{2}}}, \quad \xi \in \mathbb{R}^n, 
      \end{equation}   
      where $C(\beta,n)$ is a positive constant, and $K_\gamma$ is the modified Bessel function of order $\gamma$. 
      
    It is possible to express $(-\Delta)^{\alpha/2}u$ in terms of 
    the Gaussian hypergeometric function 
    $$F(a,b,c,x)= _{2}{\!\!}F_{1}(a,b,c,x)$$ discussed in the next section. For radially symmetric functions $f(r)$, where $r=|x|$, the Fourier transform formula 
   can be stated in the form:
       \begin{equation}\label{radial-fourier}
       \hat f(s)= (2 \pi)^{n/2} \int_0^{+\infty} \frac{J_{\frac{n-2}{2}}(sr)}{(sr)^{\frac{n-2}{2}}} f(r) \, r^{n-1} \, dr, 
       \end{equation} 
       where $s=| \xi |$.
       Hence, from the preceding formulas we deduce: 
              \begin{equation}\label{radial-fractional}
           -(-\Delta)^{\alpha/2}u(s) = C(\alpha, \beta, n) \, s^{\frac{2-n}{2}}\int_0^{+\infty} K_{\frac{n-\beta}{2}}(r) \, J_{\frac{n-2}{2}}(sr) \, r^{\frac{\beta}{2}+\alpha} \, dr,     
            \end{equation}   
            where $C(\alpha, \beta, n)$ is a positive constant.    
            
            Using the well-known integral involving the product of Bessel functions $J_\gamma$ and modified Bessel functions $K_\gamma$, we obtain the explicit formula for $(-\Delta)^{\alpha/2}u$ (see e.g., \cite{MOS}, Sec. 3.8, p. 100): 
           \begin{equation}\label{explicit}        
           -(-\Delta)^{\alpha/2}u (x)= C(\alpha, \beta, n) \, F(a,b,c, -|x|^2), 
           \quad x \in \mathbb{R}^n,
      \end{equation}  
where $C(\alpha, \beta, n)$ is a positive constant. The parameters $a$, $b$, and $c$ are given by:  
       \begin{equation}\label{parameters}  
a= \frac{n+\alpha}{2}, \quad b= \frac{\alpha+\beta}{2}, \quad c= \frac{n}{2},
   \end{equation} 
   where 
 \begin{equation}\label{parameters1}  
 \alpha = \frac{2k}{k+1}, \quad \beta =\frac{n}{k}-2, \quad n\ge 2k,  
 \quad k \ge 2.
   \end{equation}  
   Notice that $a>c>b>0$ and $a>b+1$ ($a=b+1$ when $k=1$) in our case. To verify (\ref{iterated}), we need 
   to demonstrate: 
      \begin{equation}\label{k-power} 
      (-\Delta)^{\alpha/2} \left [ F(a,b,c, -|x|^2) \right ]^k \ge 0. 
      \end{equation}  
      
      Let 
            \begin{equation}\label{k-function} 
      \phi(x) = F(a,b,c, -|x|^2), \quad x \in \mathbb{R}^n.
       \end{equation} 
    By (\ref{f-asymp}), $\phi$ clearly satisfies condition (\ref{ineFFV}).    Invoking Theorems~\ref{s-subharm} and \ref{charachter_s_subharmonic_int}, we see that (\ref{k-power}) follows from condition (\ref{cond-2}): 
         \begin{equation*}   
          ([\phi(r)]^k)'' + \frac{n-\alpha+1}{r}  ([\phi(r)]^k)' \le 0, \quad r \ge 0.
           \end{equation*}       
      By direct differentiation, 
              \begin{equation*}
       ([\phi(r)]^k)'' + \frac{n-\alpha+1}{r}([\phi(r)]^k)'   = k [\phi(r)]^{k-1} \left [ 
       \phi''(r) + (k-1) \, \frac{ [\phi'(r)]^2}{\phi(r)} +  \frac{n-\alpha+1}{r} \phi'(r)
       \right ]. 
            \end{equation*} 
 Hence it suffices to verify the inequality 
\begin{equation}\label{k-ineq}
  \phi''(r) + (k-1) \, \frac{ [\phi'(r)]^2}{\phi(r)} +  \frac{n-\alpha+1}{r} \phi'(r)\le 0, \quad r>0. 
    \end{equation} 
    Since $\phi(r)=F(-r^2)$, where $F(x) = F(a,b,c,x)$, we have 
    $\phi'(r)=-2r F'(-r^2)$ and $\phi''(r)=4r^2 F''(-r) -2F'(-r^2)$. It follows that the preceding inequality 
    is equivalent to: 
     \begin{equation}\label{k-inequa1}
  F''(-r^2) + (k-1) \, \frac{ [F'(-r^2)]^2}{F(-r^2)} -  \frac{n-\alpha+2}{2r^2} F'(-r^2)\le 0, \quad r>0. 
    \end{equation} 
     Letting $x=r^2$, $x > 0$, and noticing that $ \frac{n-\alpha+2}{2}=2c -a+1$ by (\ref{parameters}), 
     we rewrite the preceding inequality as: 
      \begin{equation}\label{k-x}
  F''(-x) + (k-1) \, \frac{ [F'(-x)]^2}{F(-x)} -  \frac{2c-a+1}{x} F'(-x)\le 0, \quad x>0. 
    \end{equation} 
  Using the hypergeometric equation (\cite{AAR}, p. 75):
      \begin{equation}\label{hyper-geo-eq}
 x(1+x) F''(-x) - (c+(a+b+1)x) \, F'(-x) + ab  F(-x)=0, 
    \end{equation} 
 we eliminate the second derivative in (\ref{k-x}):
        \begin{equation*}
        (k-1) [F'(-x)]^2 +\left [ \frac{c+(a+b+1)x}{x(1+x)} - 
     \frac{2c-a+1}{x}     \right ] F'(-x) \, F(-x) - \frac{a b}{x(1+x)} [F(-x)]^2 \le 0. 
           \end{equation*} 
           
           Letting  $\lambda = \frac{F'(-x)}{F(-x)}$, where $x \ge 0$, we 
           reduce the preceding estimate to the quadratic inequality 
             \begin{equation}\label{hyper-ineq3}
             x(1+x)(k-1) \lambda^2 + [(2a+b-2c)x +a-c-1] \, \lambda 
      - a b\le 0,  
                    \end{equation}    
 for $x > 0$.  From the differentiation formula (\cite{AAR}, p. 94), 
   \begin{equation}\label{derivative}
F'(a,b,c, x)= \frac{ab}{c} F(a+1,b+1,c+1, x),  
     \end{equation} 
it follows: 
   \begin{equation}\label{lambda}
\lambda= \frac{ab}{c} \frac{F(a+1,b+1,c+1, -x)}{F(a,b,c, -x)}>0, \quad x > 0. 
     \end{equation} 
     Solving inequality (\ref{hyper-ineq3}) for $\lambda>0$, we obtain: 
       \begin{equation}\label{hyper-ineq4a}
             0< \lambda \le   \frac{-(d_1 x - d_2) + \sqrt{ (d_1x - d_2)^2 + d_3 x(1+x)} }{2(k-1)x(1+x)}, \quad x \ge 0. 
                    \end{equation}  
where 
\begin{equation}\label{parameters2} 
d_1=2a +b-2c, \quad d_2=c-a+1, \quad d_3=   4  a b (k-1) = 4a (2c-a-b)
 \end{equation}  
 by (\ref{parameters1}). 
   Equivalently,     
       \begin{equation}\label{hyper-ineq4}
             0< \lambda \le   \frac{2ab}{d_1 x - d_2 + \sqrt{ (d_1x - d_2)^2 + d_3 x(1+x)} }, \quad x > 0, 
                    \end{equation}    
Using (\ref{lambda}), 
     we see that  (\ref{hyper-ineq3}) is equivalent to: 
     \begin{equation}\label{hyper-ineq5}
             \frac{F(a,b,c, -x)}{F(a+1,b+1,c+1,-x)} \ge    \frac{d_1 x - d_2 + \sqrt{ (d_1x - d_2)^2 + d_3 x(1+x)} } {2c}, \quad x > 0. 
\end{equation}

We next use Pfaff's transformation (\cite{AAR}, Theorem 2.2.5, p. 68): 
  \begin{equation}\label{pfaff}
F(a,b,c, y)= (1-y)^{-b} F(c-a,b,c, \frac {y}{y-1}),  \quad \text{for} \, \, \, y<1. 
     \end{equation} 
     Letting $y=-x$, $x\ge 0$, and applying (\ref{pfaff}) to both 
     $F(a,b,c, -x)$ and $F(a+1,b+1,c+1,-x)$, we get 
  \begin{equation}\label{pfaff1}   
    \frac{F(a,b,c, -x)}{F(a+1,b+1,c+1,-x)}   = (1+x)  \frac{F(c-a,b,c, \frac {x}{x+1})}{F(c-a,b+1,c+1,\frac {x}{x+1})}, \quad x > 0. 
\end{equation} 
Letting $t=\frac {x}{x+1}$, we rewrite (\ref{hyper-ineq5}) 
in the equivalent form:
   \begin{equation}\label{hyper-ineq6}
             \frac{F(c-a,b,c, t)}{F(c-a,b+1,c+1,t)} \ge   
              \frac{d_1 t - d_2(1-t) + \sqrt{ (d_1 t  - d_2(1-t) )^2 +d_3 t} } {2c}, \quad 0 < t \le 1, 
\end{equation}    
where $d_1, d_2, d_3$ are given by (\ref{parameters2}). We note that since 
$(c-a)+b< c$ and $(c-a)+ (b+1)<c+1$, it follows that the 
both  hypergeometric functions in the ratio on the left-hand side 
are finite at $t=1$. By  (\ref{parameters-2}),  
   \begin{equation}\label{f-1}
F(c-a,b,c,1)= \frac{\Gamma(c) \Gamma(a-b)}{\Gamma(a) \Gamma(c-b)}, \quad F(c-a,b+1,c+1,1)= \frac{\Gamma(c+1) \Gamma(a-b)}{\Gamma(a+1) \Gamma(c-b)}. 
\end{equation}  
Consequently, 
  \begin{equation}\label{f-2}
\frac{F(c-a,b,c,1)}{F(c-a,b+1,c+1,1)}= \frac{a}{c}. 
\end{equation}  
In fact, since $a>b+1$ for $k\ge 2$, both hypergeometric functions in the ratio are twice continuously differentiable for $t \le 1$, including the endpoint $t=1$ when (see  
\cite{AAR}, Theorem 2.3.2, p. 78). 

To prove (\ref{hyper-ineq6}), we express it in the form:
 \begin{equation}\label{f-g-ineq}
f(t) \ge g(t), \quad 0\le t \le 1, 
\end{equation} 
where  $f$ is the Gaussian ratio: 
\begin{equation}\label{f-convex}
f(t)= \frac{F(c-a,b,c, t)}{F(c-a,b+1,c+1,t)}, \quad -\infty < t \le 1,
\end{equation}  
and 
\begin{equation}\label{def-g}   
   g(t)= \frac{d_1 t - d_2(1-t) + \sqrt{ (d_1 t  - d_2(1-t) )^2 + d_3 t} } {2c},  \quad t\ge 0. 
   \end{equation}

The proof of (\ref{f-g-ineq}) relies on the important observation  that  $f$ is convex (in fact, even logarithmically convex, as is proved in the next section)  while $g$ 
is concave:  

\begin{lemma}\label{lemma-f-g}  The function $g$ defined by (\ref{def-g}) is concave on $[0,+\infty)$. \end{lemma}

\begin{proof} Let $h(t)= \sqrt{A t^2 + B t +C}$, then 
\begin{equation*}   
   h'(t)= \frac{2 At +B}{2 \sqrt{A t^2 + Bt +C}},  \quad \text{and} \quad 
     h''(t)= \frac{4AC-B^2}{4 (A t^2 + Bt +C)^{\frac 3 2}}. 
   \end{equation*} 
   Hence, $h$ is concave if $B^2-4AC\ge 0$. 
  Let $A t^2 + B t +C= (d_1 t  - d_2(1-t) )^2 + d_3 t $,
  where   $d_1, d_2, d_3$ are defined by  (\ref{parameters2}).  Then clearly, 
    $A=(d_1-d_2)^2$, $B=d_3-2 d_2(d_1+d_2)$,  $C=d_2^2$, 
and  $A t^2 + B t +C\ge 0$ for 
  $t\ge 0$ since $d_3\ge 0$. Clearly, 
  \begin{equation*}  
  B^2-4AC=[d_3-2 d_2(d_1+d_2)]^2 - 4 (d_1+d_2)^2 d_2^2=
  d_3  [d_3- 4d_2 (d_1+d_2)],
     \end{equation*} 
      where  
      \begin{equation*}  
      d_3- 4d_2 (d_1+d_2)=a(2c-a-b)- (c-a+1)(a+b-c+1)= 
      (c+1)(c-b-1) \ge 0. 
       \end{equation*} 
       Notice that $c- b-1\ge 0$. Indeed, from  (\ref{parameters1}) it follows that 
    $n(k-1)\ge 2k$ since  $n \ge 2k$ and $k \ge 2$.  Hence, 
   \begin{equation}\label{f-2b}  
    c-b-1= \frac{n(k-1)-2k + \alpha}{2k}> \frac{n(k-1)-2k}{2k} \ge 0.   
   \end{equation} 
     Thus, $B^2-4AC > 0$, and consequently $h$ is concave. 
     Since $g''=\frac{1} {2c}h''<0$, it follows that $g$  is concave as well.
   \end{proof}

Since $f$ is convex by Theorem~\ref{log-conv}  below, and $g$ is concave by Lemma~\ref{lemma-f-g}, it follows that the graph 
of $f$ lies above the tangent line at $t=1$, while the graph 
of $g$ lies below the tangent line at $t=1$. In other words, 
$f(t) \ge f(1) + (t-1) f'(1)$, while $g(t)\le g(1) + (t-1) g'(1)$ for $0\le t\le 1$. Recall that by  (\ref{f-2}), 
$f(1)= \frac {a}{c}$. Notice that $g(1)= \frac {a}{c}$ as well. Indeed, 
\begin{equation*}  
   g(1)= \frac{d_1 + \sqrt{ d_1^2 + d_3} } {2c}, 
   \end{equation*}
   where 
   \begin{equation*} 
d_1^2 +d_3=a^2-2a(2c-a-b) + (2c-a-b)^2 + 4a(2c-a-b)
\end{equation*} 
 \begin{equation*} 
=a^2+2a(2c-a-b) + (2c-a-b)^2
=(2c-b)^2. 
\end{equation*} 
Since $2c>b$ by  (\ref{parameters1}), we obtain: 
\begin{equation}\label{sq-root}
\sqrt{ d_1^2 + d_3} = 2c-b, 
\end{equation}
and consequently, by (\ref{parameters2}),
\begin{equation*}
g(1)= \frac{d_1+ 2c-b} {2c} =\frac{2a+b-2c +2c-b}{2c}=\frac{a}{c}. 
\end{equation*} 

Thus, $f(1)=g(1)$, and to verify (\ref{f-g-ineq}), it remains to show that $ f'(1) \le  g'(1)$. Let us deduce the formula:
\begin{equation}\label{f-2f}
f'(1)= \frac{a(a-c)}{c(a-b-1)}. 
\end{equation} 
Invoking the differentiation formula 
(\ref{derivative}), we obtain: 
  \begin{equation*}  
f'(t)=\frac{(c-a)b}{c} \, \frac{F(c-a+1,b+1,c+1,t)}
{F(c-a,b+1,c+1,t)} - 
 \frac{(c-a)(b+1)}{c+1} \frac{F(c-a+1,b+2,c+2,t) \, F(c-a,b,c,t)}{F(c-a,b+1,c+1,t)^2}.  
\end{equation*}
 Using  (\ref{f-2}), we deduce from the preceding formula: 
  \begin{equation*} 
 f'(1)=\frac{(c-a)b}{c} \, \frac{F(c-a+1,b+1,c+1,1)}
{F(c-a,b+1,c+1,1)} - 
 \frac{a(c-a)(b+1)}{c(c+1)}\frac{F(c-a+1,b+2,c+2,1)}
 {F(c-a,b+1,c+1,1)}.   
 \end{equation*}
 Applying Gauss's formula (\ref{f-1}) for $F(a,b,c,1)$  (\cite{AAR}, Theorem 2.2.2, p. 66), we obtain
  \begin{equation*} 
f'(1) =\frac{(c-a)b}{c}  \frac{\Gamma(c+1) \Gamma(a-b-1) \Gamma(a+1) \Gamma(c-b)}{\Gamma(a) \Gamma(c-b) \Gamma(c+1) \Gamma(a-b)} 
 \end{equation*}
   \begin{equation*} 
-   \frac{a(c-a)(b+1)}{c(c+1)} 
  \frac{ \Gamma(c+2) \Gamma(a-b-1)  \Gamma(a+1) 
 \Gamma(c-b)}
{\Gamma(a+1) \Gamma(c-b)\Gamma(c+1) \Gamma(a-b) }  
\end{equation*} 
 \begin{equation*} 
 =\frac{(c-a)a b}{c(a-b-1)}  - 
  \frac{a(c-a)(b+1)}{c (a-b-1)}= \frac{a(a-c)}{c(a-b-1)}. 
   \end{equation*} 
This proves (\ref{f-2f}).

On the other hand, by direct differentiation, 
\begin{equation}\label{f-2c}
g'(t)= \frac{1}{2c} \left ( d_1+ d_2 + \frac{2(d_1+d_2)(d_1 t-d_2(1-t))+d_3}{2 \sqrt{(d_1 t-d_2(1-t))^2 +d_3}}\right ). 
\end{equation} 
Hence, 
\begin{equation}\label{f-2d}
g'(1)= \frac{1}{2c} \left ( d_1+ d_2 + \frac{2(d_1+d_2)d_1+d_3}{2 \sqrt{d_1^2 +d_3}}\right ). 
\end{equation} 
We next show that, after simplification, we get:
\begin{equation}\label{f-2e}
g'(1)= 
\frac{a(c+1)}{c(2c-b)}. 
\end{equation} 
To prove this, recall that $\sqrt{d_1^2 +d_3} = 2c-b$ by (\ref{sq-root}), and $d_1=2a+b-2c=a-(2c -a-b)$, $d_1+d_2=a+b-c+1$,  and $d_3=4a(2c-a-b)$ by (\ref{parameters2}). It follows: 
\begin{equation*}
g'(1)= \frac{1}{2c} \left (a+b-c+1+ \frac{(2a+b-2c)(a+b-c+1)+ 2a(2c-a-b)}{2c-b}\right ). 
\end{equation*} 
Simplifying further the right-hand side, we rewrite it as follows: 
\begin{equation*}
\frac{(2c-b)(a+b-c+1)+ (2a+b-2c)(a+b-c+1)+ 2a(2c-a-b)}{2c(2c-b)}
\end{equation*} 
\begin{equation*}
 =  \frac{2a(a+b-c+1)+ 2a(2c-a-b)}{2c(2c-b)} = 
 \frac{a(c+1)}{c(2c-b)}. 
\end{equation*} 
This proves (\ref{f-2e}). 

Now it is not difficult to see that $g'(1) \ge f'(1)$, i.e.,
  \begin{equation}\label{f-2a} 
 \frac{a(a-c)}{c(a-b-1)} \le \frac{a(c+1)}{c(2c-b)}.
 \end{equation}  
 Indeed, by (\ref{parameters1}),  $2c>b>0$ and $a-b-1=\frac{n(k-1)}{2k}>0$ for $k \ge 2$. Hence (\ref{f-2a}) 
 is equivalent to:
  \begin{equation*}  
   (a-c)(2c-b)- (c+1)(a-b-1) \le 0. 
   \end{equation*}
 By factoring, we rewrite this as:
    \begin{equation*}  
    (a-c)(2c-b)- (c+1)(a-b-1) = ac-2c^2 -ab+2bc + c-a+b+1)
       \end{equation*}
         \begin{equation*} 
    =c(a-2c)-b(a-2c) -(a-2c) - (c-b-1)=(a-2c-1)(c-b-1) \le 0. 
   \end{equation*}
 By (\ref{f-2b}), $c-b-1\ge 0$. 
   On the other hand, since $1<\alpha<2$, we have: 
     \begin{equation*}  
   a-2c-1=-\frac{n+2-\alpha}{2}<0,    
   \end{equation*}
   This proves (\ref{f-2a}), and consequently  $g'(1)\ge f'(1)$, which yields (\ref{f-g-ineq}).

       These computations have been verified using \textit{Mathematica}, 
which was also used to check that inequality (\ref{hyper-ineq6}) holds for many concrete values of $k$ and $n$. In the next section, we will complete the proof of this estimate by showing that $f(t)$ is convex if $t \le 1$. 

\begin{remark} 
When $k=1$ and $\alpha=1$, the function $u(x)=-(-\Delta)^{\tfrac 1 2} 
(1+|x|^2)^{2-n}$, 
$n \ge 3$, is $\tfrac 1 2$-subharmonic, i.e. (\ref{iterated}) holds, but inequality (\ref{cond-2}) with $s = \tfrac 1 2$ fails 
for $r$ large. In other words,  (\ref{cond-2}) is sufficient but 
not necessary for a $C^2$ radial function to be $s$-subharmonic. 
\end{remark}

Indeed, if $k=1$, then $a=b+1$, and consequently by 
(\ref{f-1}) and (\ref{f-2f}), it follows that $f'(1)=+\infty$, where 
 \begin{equation*}  
   f(t)= \frac{F(c-a,b,c, t)}{F(c-a,b+1,c+1,t)}, \quad -\infty < t \le 1, 
   \end{equation*} 
Note that inequality (\ref{k-x}), 
and hence (\ref{hyper-ineq3}), is equivalent to (\ref{hyper-ineq6}),  
i.e.,  $f(t)\ge  g(t)$ for $t \le 1$, 
   where $g(t)$ is defined by (\ref{def-g}) with $d_3=0$. Since 
   $f(1)=g(1)=\frac a c$, and $g'(1) < +\infty$ by (\ref{f-2e}), this contradicts $f'(1)=+\infty$.

\section{Logarithmic convexity of the ratio of hypergeometric functions}\label{logconvexity} 

Let $F(a,b,c,x)= \! _{2}F_{1}(a,b,c,x)$ denote the hypergeometric function. We refer to \cite{AAR} and \cite{MOS} for the relevant theory. 
 For $c>b>0$,  it can be defined by 
(\cite{AAR}, p. 65): 
   \begin{equation}\label{hypergeo}
   F(a,b,c,z)= \frac{\Gamma(c)}{\Gamma(b) \Gamma(c-b)} 
   \int_0^1  \frac{s^{b-1}(1-s)^{c-b-1}}{(1-sz)^a} ds, 
     \end{equation} 
     in the complex plane with a cut from $1$ to $\infty$ along the 
     positive real axis. We will show in this section 
     that the Gaussian ratio $\varphi(x)=\frac{F(a,b,c,x)}{F(a,b+1,c+1,x)}$ is logarithmically 
     convex, and hence convex, for $x\in (-\infty, 1)$ if $c>b>0$ and $-1<a<0$. Since $a+b<c$, it follows by (\ref{f-1})
     that $\varphi$ is continuous at $x=1$ as well, and hence $\varphi$  is convex in $(-\infty, 1]$, including the end-point. Applying 
     this to $f(x)= \frac{F(c-a,b,c, x)}{F(c-a,b+1,c+1,x)}$ where 
     $a$, $b$, $c$ are given by (\ref{parameters}) so that $c>b>0$, and $c-a=-\frac{\alpha}{2} \in (-1,0)$, we will conclude the proof of Theorem~\ref{k-convex-example}.

     The proof of the following theorem employs a method developed recently in \cite{KS1} (see also \cite{KS2}) 
     to prove monotonicity and log-concavity in parameters 
     $a, b, c$  of generalized hypergeometric 
     functions and their ratios. 
          
     \begin{theorem}\label{log-conv} Let $-1<a<0$, and $c>b>0$. 
     Then the function 
     $\varphi (x) = \frac{F(a,b,c,x)}{F(a,b+1,c+1,x)}$ 
     is logarithmically convex in $(-\infty, 1)$. 
     \end{theorem}
     
      \begin{proof} Let $x\in (-\infty,1)$. Let $f_1(x) = F(a,b,c,x)$ and $f_2(x) = F(a,b+1,c+1,x)$. Then $(\log \varphi)'(x) = \frac{f_1'(x)}{f_1(x)}-\frac{f_2'(x)}{f_2(x)}$, and 
    \begin{equation*}
      (\log \varphi)''(x) = \frac{f_1''(x)}{f_1(x)}-\frac{f_2''(x)}{f_2(x)} - 
    \left (  \left [\frac{f_1'(x)}{f_1(x)}\right]^2- \left [\frac{f_2'(x)}{f_2(x)}\right]^2\right )
   \end{equation*} 
       \begin{equation*}
       =\frac{f_1''(x) f_2(x) - f_2''(x) f_1(x)}{f_1(x) \, f_2(x)}-
       \frac{[ f_1'(x) f_2(x) - f_2'(x) f_1(x)] \cdot 
       [f_1'(x) f_2(x) + f_2'(x) f_1(x)]}{[f_1(x) \, f_2(x)]^2}. 
          \end{equation*} 
  Clearly, $f_1(x)>0$ and $f_2(x)>0$ by (\ref{hypergeo}). Since $a<0$ and $b>0, c>0$, it follows by the differentiation formula (\ref{derivative}): 
       \begin{equation*}
       f_1'(x)=\frac{a b}{c} \, F(a+1,b+1,c+1,x)<0, \quad  f_2'(x)=\frac{a (b+1)}{c+1} \, F(a+1,b+2,c+2,x)<0,  
         \end{equation*} 
        and consequently, 
         \begin{equation*}
    f_1'(x) f_2(x) + f_2'(x) f_1(x)  <0,  \quad -\infty < x<1. 
         \end{equation*} 
 
         Hence to show that $(\log \varphi)''(x)>0$, it suffices to prove the following inequalities: 
         \begin{equation}\label{mon-conv}
f_1'(x) f_2(x) - f_2'(x) f_1(x)>0, \quad f_1''(x) f_2(x) - f_2''(x) f_1(x)>0, \quad -\infty < x<1. 
\end{equation} 

It will be more convenient to make a substitution $t=-x$, 
and  work with $g_1(t) = f_1(-t)$, $g_2(t)=f_2(-t)$ for $t \in (-1, +\infty)$. Then the preceding inequalities are equivalent to:
  \begin{equation}\label{mon-conv1}
g_1'(t) g_2(t) - g_2'(t) g_1(t)<0, \quad  g_1''(t) g_2(t)-g_2''(t) g_1(t)>0, \quad -1<t<+\infty. 
\end{equation} 

We next employ the integral representation (\ref{hypergeo}) to express 
both $g_1$ and $g_2$ in the  form of Stieltjes-type transformations:
        \begin{equation*}
     g_1(t) =   \int_0^1  s^{b-1} \, (1+st)^{-a} \,  w(s) \, ds, 
     \quad    g_2(t) = \frac {c}{b} \int_0^1  s^{b} \, (1+st)^{-a} \,   w(s) \, ds, \quad - 1<t<+\infty,
         \end{equation*}   
  where 
    \begin{equation*}
   w(s)=\frac{\Gamma(c)}{\Gamma(b) \Gamma(c-b)} (1-s)^{c-b-1}, \quad 0<s<1,
     \end{equation*}   
     is a positive weight function. Then we have:
       \begin{equation*}
     g_1'(t) =  -a \int_0^1  s^{b} \, (1+st)^{-a-1} \,  w(s) \, ds, 
     \quad    g_2'(t) = -a \,  \frac {c}{b} \int_0^1  s^{b+1} \, (1+st)^{-a-1} \,   w(s) \, ds, 
              \end{equation*}   
      \begin{equation*}
     g_1''(t) =  a(a+1) \int_0^1  s^{b+1} \, (1+st)^{-a-2} \,  w(s) \, ds, 
     \quad    g_2''(t) =  a(a+1)  \, \frac {c}{b} \int_0^1  s^{b+2} \, (1+st)^{-a-2} \,   w(s) \, ds, 
         \end{equation*}   
     where $- 1<t<+\infty$. 
     
     To prove the first inequality in (\ref{mon-conv1}), i.e., $g_1'(t) g_2(t)<g_2'(t) g_1(t)$, we rewrite it in the equivalent form: 
        \begin{equation*}
   -a \, \frac {c}{b}   \left(   \int_0^1  s^{b} \, (1+st)^{-a-1} \,  w(s) \, ds\right)   \,  
        \left( \int_0^1  s^{b} \, (1+st)^{-a} \,   w(s) \, ds \right)
                \end{equation*}  
                   \begin{equation*}     
                 < 
     -a \,  \frac {c}{b} \left (\int_0^1  s^{b+1} \, (1+st)^{-a-1} \,   w(s) \, ds\right) \,  \left ( \int_0^1  s^{b-1} \, (1+st)^{-a} \,  w(s) \, ds\right). 
        \end{equation*}  
        Similarly, the second inequality in   (\ref{mon-conv1}), i.e., $g_1''(t) g_2(t)>g_2''(t) g_1(t)$, is equivalent to: 
         \begin{equation*}
   a (a+1)\, \frac {c}{b}   \left(   \int_0^1  s^{b+1} \, (1+st)^{-a-2} \,  w(s) \, ds\right)   \,  
        \left( \int_0^1  s^{b} \, (1+st)^{-a} \,   w(s) \, ds \right)
                \end{equation*}  
                   \begin{equation*}     
                 >
     a (a+1)\,  \frac {c}{b} \left (\int_0^1  s^{b+2} \, (1+st)^{-a-2} \,   w(s) \, ds\right) \,  \left ( \int_0^1  s^{b-1} \, (1+st)^{-a} \,  w(s) \, ds\right). 
        \end{equation*}  
     Fix  $t\in (-1,+\infty)$. 
   Since $c>b>0$, and $-1<a<0$, the first inequality above can be 
        expressed in the form (see \cite{KS1}, p. 341):
        \begin{equation}\label{first-der}
             \left(   \int_0^1  h(s) \,  p(s) \, ds\right)   \,  
        \left( \int_0^1  q(s) \, p(s) \, ds \right)  
        < \left (\int_0^1  q(s) \, h(s) \, p(s) \, ds\right) \,  \left ( \int_0^1  p(s) \, ds\right),
        \end{equation}   
        where $h(s) = \frac{s}{1+st}$, $p(s) = s^{b-1} (1+st)^{-a} w(s)$, $q(s)=s$ and  $w(s)=\frac{\Gamma(c)}{\Gamma(b) \Gamma(c-b)} (1-s)^{c-b-1}$. 
        
        Analogously, the second inequality is equivalent to: 
               \begin{equation}\label{second-der}
  \left(   \int_0^1  h_1(s) \, p(s) \, ds\right)   \,  
        \left( \int_0^1  q(s) \, p(s) \, ds \right )   
                 <
  \left (\int_0^1  q(s) \, h_1(s) \,   p(s) \, ds\right) \,  \left ( \int_0^1 p(s) \, ds\right), 
        \end{equation}  
        where $h_1(s) = \frac{s^2}{(1+st)^2}$, $p(s)$, $q(s)$ and  $w(s)$ are as above. Since $p(s) > 0$, and the functions 
        $q(s)$, $h(s)$, and $h_1(s)$ are increasing on $(0,1)$ for any fixed $t > -1$, both (\ref{first-der}) and (\ref{second-der}) 
       follow from Chebyshev's inequality for monotone functions (see \cite{HLP}).  \end{proof}
    
    \begin{remark} The condition $c>b>0$ in Theorem~\ref{log-conv} 
     can be extended to $c>\min(a,b)>0$ using the symmetry 
     of $F(a,b,c,x)$ in $a$ and $b$.  
     \end{remark}

      \begin{remark} 
      Theorem~\ref{log-conv} holds  for the ratio of  
      hypergeometric functions $_{q+1}F_{q}\left ( (a, \mathbf{b}),\mathbf{c},x\right)$ 
      in place of $_{2}F_{1}(a,b,c,x)$, where $-1<a<0$ and 
      $c_k> b_k>0$, $k=1, 2, \ldots, q$, $q \ge 2$. 
     \end{remark}
     The proof of the logarithmic convexity of the ratio of $_{q+1}{\!}F_{q}$ is similar to the 
     proof of Theorem~\ref{log-conv} using the corresponding 
     Stieltjes type integral representation (\cite{KS1}, Lemma 1).

       \begin{corollary}\label{conv2} Let $0<a-c<1$, and $c>b>0$. 
     Then the ratio  
     $ \frac{F(a,b,c,x)}{F(a+1,b+1,c+1,x)}$ 
     is convex in $(-\infty, 1)$. 
     \end{corollary}
     
     \begin{proof} Let  $\varphi_1 (-x) = \frac{F(a,b,c,x)}{F(a+1,b+1,c+1,x)}$, $x<1$.  
     By  (\ref{pfaff1}),  
   \begin{equation*} 
   \varphi_1 (x) =       (1+x) \, \varphi_0(\frac{x}{x+1}), \quad x>-1,
    \end{equation*}  
             where     
              \begin{equation*} 
          \varphi_0 (y) =   \frac{F(c-a,b,c,y)}{F(c-a,b+1,c+1,y)}, \quad y<1.  
         \end{equation*}   
        We deduce: 
                            \begin{equation*} 
          \varphi_1' (x) =     \varphi_0(\frac{x}{x+1}) +\varphi_0'(\frac{x}{x+1}) \frac{1}{1+x}, \quad  \varphi_1'' (x) =  
          \varphi_0''(\frac{x}{x+1}) \frac{1}{(x+1)^3} \ge 0,  
         \end{equation*}  
    where $\varphi_0$ is logarithmically convex by  Theorem~\ref{log-conv}, and consequently convex, so that $\varphi_0''(y)\ge 0$, $y<1$.   Then, obviously, $ \varphi_1 (-x)$ is convex  as well. 
                           \end{proof}

{\bf Acknowledgements.} The authors wish to express their thanks for the hospitality during  their respective visits to the mathematics departments of the 
University of Missouri and Universit\`a di Bologna.\\
F.F. was partially supported by PRIN project: 
 `Metodi di viscosit\`a, geometrici e di controllo per modelli diffusivi nonlineari," the GNAMPA project: `Equazioni non lineari su variet\`a: propriet\`a qualitative e classificazione delle soluzioni,"  the ERC starting grant project 2011 EPSILON (Elliptic PDEs and Symmetry of Interfaces and Layers for Odd Nonlinearities), 
 and the Miller Fund at the University of Missouri.\\ 
  I.E.V. was partially supported 
by NSF grant  DMS-0901550.

\end{document}